\documentclass{amsart}

\usepackage{amssymb,amscd,amsthm,amsxtra}
\usepackage{dsfont,latexsym}
\usepackage{url}

\usepackage{mathrsfs}
\renewcommand{\mathcal}{\mathscr}

\usepackage{color}
\definecolor{citation}{rgb}{0.2,0.5,0.2}
\definecolor{formula}{rgb}{0.1,0.2,0.5}
\definecolor{url}{rgb}{0,0.2,0.7}
\usepackage[colorlinks=true,linkcolor=formula,urlcolor=url,citecolor=citation]{hyperref}

\newtheorem{theorem}{Theorem}[section]
\newtheorem{corollary}[theorem]{Corollary}
\newtheorem{lemma}[theorem]{Lemma}
\newtheorem{prop}[theorem]{Proposition}

\theoremstyle{definition}
\newtheorem{defn}[theorem]{Definition}
\theoremstyle{remark}
\newtheorem{rem}[theorem]{Remark}
\numberwithin{equation}{section}

\newcommand{\R}{{\mathds R}}

\newlength{\defbaselineskip}
\newcommand{\setlinespacing}[1]
           {\setlength{\baselineskip}{#1 \defbaselineskip}}

\title
[A geometric inequality and a symmetry result.]
{A geometric inequality and a symmetry result for elliptic systems involving the fractional Laplacian.}

\author[S. Dipierro]{Serena Dipierro}
\address[Serena Dipierro]{SISSA, Sector of Mathematical Analysis
\\ Via Bonomea, 265 
\\ 34136 Trieste, Italy}
\email{\href{mailto:serydipierro@yahoo.it}{serydipierro@yahoo.it}}

\author[A. Pinamonti]{Andrea Pinamonti}
\address[Andrea Pinamonti]{Dipartimento di Matematica, Universit\`a di Padova \\ Via Trieste 63}
\email{\href{mailto:pinamonti@science.unitn.it}{pinamonti@science.unitn.it}}

\begin{document}
\vskip .2truecm

\keywords{Elliptic systems, fractional Laplacian, monotone solutions, stable solutions, phase separation, Poincar\'{e}-type inequality  \vspace{1mm}}

\thanks{The first author has been supported by FIRB ``Project Analysis and Beyond''. The second author has been supported by MIUR, Italy, GNAMPA of INDAM and by Fondazione CaRiPaRo Project ``Nonlinear Partial Differential Equations: models, analysis, and control-theoretic problems``}

\begin{abstract} 
We study the symmetry properties for solutions of elliptic systems of the type
\begin{eqnarray*}
\left\{ 
\begin{array}{ll} 
    \left(-\Delta\right)^{s_1} u = F_1(u, v),         \\
    \left(-\Delta\right)^{s_2} v= F_2(u, v),
 \end{array} 
\right.
\end{eqnarray*} 
where~$F\in C^{1,1}_{loc}\left(\R^2\right)$,~$s_1,s_2\in (0,1)$ and 
the operator~$\left(-\Delta\right)^s$ is the so-called fractional Laplacian. 
We obtain some Poincar\'e-type formulas for the~$\alpha$-harmonic extension in the half-space, 
that we use to prove a symmetry result both for stable and for monotone solutions.
\end{abstract}

\setcounter{tocdepth}{2}

\maketitle
%{\small \tableofcontents}

\setlinespacing{1.09}

\section{Introduction}
In this paper we deal with the following system in~$\R^n$
\begin{eqnarray}\label{sistpart}
\left\{ 
\begin{array}{ll} 
    \left(-\Delta\right)^{s_1} u = F_1(u, v),         \\
    \left(-\Delta\right)^{s_2} v= F_2(u, v),
 \end{array} 
\right.
\end{eqnarray} 
where~$s_1,s_2\in(0,1)$, $F\in C^{1,1}_{loc}\left(\R^2\right)$, and~$F_1$ and~$F_2$ 
denote the derivatives of~$F$ with respect to the first and the second variable respectively. 

As customary, we denote by~$\left(-\Delta\right)^s$, with~$s\in (0,1)$, 
the fractional Laplacian. We recall that it can be defined, up to a multiplicative constant, by the following formula
\begin{equation}\label{fraclapl}
\left(-\Delta\right)^s u(x) = P.V. \int_{\R^n}\frac{u(x)-u(y)}{|x-y|^{n+2s}}\, dy, 
\end{equation}
where P.V. denotes the Cauchy principal value 
(see \cite{DPV} for the definition and for further details). 
 
If one looks at the quantity~$\left(-\Delta\right)^s$ from a distributional 
point of view, it is well-defined on every~$u$ belonging to the space 
$$ \mathcal T_s =\left\lbrace u\in\mathcal S'(\R^n) : \int_{\R^n}\frac{|u(x)|}{\left(1+|x|\right)^{n+2s}}\, dx< +\infty\right\rbrace \cap C^2_{loc}(\R^n). $$ 
We observe that, in particular, the fractional Laplacian is well-defined 
on smooth bounded functions. 

Notice that the integral in~\eqref{fraclapl} is convergent at infinity, because of 
the~$L^1$ assumption, 
and it is well-defined near the singularity, thanks to 
the assumption of~$C^2_{loc}$-regularity \footnote{We remark that for~$u\in C^2_{loc}(\R^n)$ the singular integral in~\eqref{fraclapl} is well-defined for any~$s\in (0,1)$. 
Obviously, depending on the values of~$s$, it is possible to relaxe such assumption. 
}. 

We also recall that the fractional Laplacian has a nice probabilistic interpretation, indeed 
it can be seen as the infinitesimal generator of a Levy process (see, for instance~\cite{Be} and~\cite{Va}). 

Equations containing the fractional Laplacian or more general nonlocal operators arise in several areas such as optimization, 
flame propagation and finance, see for instance~\cite{CRS, CT, DL}. 
In~\cite{ABS, Go} the authors studied phase transitions driven by fractional 
Laplacian-type boundary effects in a Gamma convergence setting. 
Finally, in~\cite{CCFS}, power-like nonlinearities for boundary reactions have also been considered.

Aim of the present paper is to prove some symmetry results for the system~\eqref{sistpart}. Similar results have been obtained in~\cite{BLWZ} for the system
\begin{eqnarray}\label{syst}
\left\{ 
\begin{array}{ll} 
    \Delta u   = uv^2,         \\
    \Delta v = vu^2, \\
    u, v>0.
 \end{array} 
\right. 
\end{eqnarray} 
A system like this arises in phase separation for multiple states of Bose-Einstein 
condensates. 
The authors proved the existence, symmetry and nondegeneracy of the solution to problem~\eqref{syst} 
in~$\R$; in particular, they showed that entire solutions are reflectionally symmetric, 
namely that there exists~$x_0$ such that~$u(x-x_0)=v(x_0-x)$. 
Moreover, they estabilished a result that may be considered the analogue of a famous 
conjecture of De Giorgi for problem~\eqref{syst} in dimension~$2$, 
that is they proved that monotone solutions of~\eqref{syst} in~$\R^2$ have one-dimensional symmetry 
under the additional growth condition
\begin{equation}\label{growth}
  u(x)+v(x) \leq C(1+|x|). 
\end{equation}
On the other hand, in~\cite{NTTV}, it has been proved that 
the linear growth is the lowest possible for solutions to~\eqref{syst}; 
in other words, if there exists~$\alpha\in (0,1)$ such that 
$$
   u(x)+v(x) \leq C(1+|x|)^{\alpha},
$$
then~$u=v\equiv 0$. 

In~\cite{BSWW} the authors proved that the above mentioned one-dimensional symmetry still holds in~$\R^2$ when the monotonicity condition is replaced by the stability of the solutions 
(which is a weaker assumption). 
Moreover, they showed that there exist solutions to~\eqref{syst} 
which do not satisfy the growth condition~\eqref{growth}, 
by constructing solutions with polynomial growth.  

In the paper~\cite{Wa} it has been proved that, for any~$n\geq 2$, 
a solution to~\eqref{syst} which is a local minimizer and satisfies the growth condition~\eqref{growth} 
has one-dimensional symmetry. 

In~\cite{FG} the authors proved that the symmetry result discovered in~\cite{BLWZ} holds
also for a more general class of nonlinearities. 

Finally, in~\cite{Di}, the author considered a class of quasilinear (possibly degenerate) elliptic systems in~$\R^n$ and proved that, under suitable assumptions, the solutions have one-dimensional symmetry, 
showing that the results obtained in~\cite{BLWZ, BSWW, FG} hold in a more general setting.
 
\bigskip

Symmetry results as the ones described above are also well-understood in the case of one equation. In particular, De Giorgi conjecture on the flatness 
of level sets of standard phase transition has been studied in dimension $n=2,3$, 
see~\cite{AAC, AC, BCN, GG, GG1}. 
Moreover, under an additional assumption on the behaviour of the solution at infinity, in~\cite{Sav} the author proved the conjecture up to dimension~$8$. 
Finally, in dimension~$n\geq9$ Del Pino, Kowalczyk and Wei constructed a solution 
to the Allen-Cahn equation which is monotone in one direction but not one-dimensional, 
see~\cite{DKW}. It is also worth noticing that an analogous of the De Giorgi conjecture has been studied for more general operators. In particular we mention~\cite{FSV}, 
where quasilinear operators of p-Laplacian and curvature type are considered, 
and~\cite{CSM}, where the authors proved a similar De Giorgi-type result for an equation involving the fractional Laplacian in dimension $n=2$ and $s=1/2$, 
see also~\cite{CaCi,CSir, CSir2} for further extensions.

In order to perform our analysis, we borrow a large number of ideas from~\cite{SV}, 
where the authors consider the following nonlocal equation
\begin{equation}\label{eqfrac} 
\left(-\Delta\right)^s u = f(u) \qquad \mathrm{in\ } \R^n, 
\end{equation}
and study the symmetry properties of the solutions. 
In particular, they prove that the analogue of the De Giorgi conjecture holds 
for equations of type \eqref{eqfrac} in dimension~$n=2$. 

The study of this nonlocal equation is based on the fact that problem~\eqref{eqfrac} 
can be reduced to the~$\alpha$-harmonic extension in the half-space.
In fact, one can consider the following boundary reaction problem for~$U=U(x,y)$, 
with~$x\in\R^n$ and~$y\in(0, +\infty)$, 
\begin{eqnarray*}
\left\{ 
\begin{array}{ll} 
div\left( y^\alpha \nabla U\right)=0, \qquad & \mathrm{on\ } \R^{n+1}_+ :=\R^n\times (0, +\infty),  \\
\lim_{y\rightarrow 0^+}\left(-y^\alpha \partial_y U\right)=f(U), \qquad & \mathrm{on\ } \R^n\times \left\lbrace 0\right\rbrace. 
 \end{array} 
\right. 
\end{eqnarray*} 
Then in~\cite{CS} the author proved that, up to a normalizing factor, 
the Dirichlet-to-Neumann 
operator~$\Gamma_\alpha :U|_{\partial\R^{n+1}_+}\mapsto -y^\alpha \partial_y U|_{\partial\R^{n+1}_+}$ 
is precisely~$\left(-\Delta\right)^{\frac{1-\alpha}{2}}$.  
This means that~$U(x,0)$ is a solution of 
$$ \left(-\Delta\right)^{\frac{1-\alpha}{2}} U(x,0)=f(U(x,0)). $$ 
Notice that the requirement~$\frac{1-\alpha}{2}=s$ in~\eqref{eqfrac} 
implies that~$\alpha\in(-1,1)$.

From a qualitative point of view, the result obtained in~\cite{CS} 
asserts that if one add a variable then one can localize 
the fractional Laplacian. 
This result has a fundamental role, for instance, in the regularity theory
for the quasigeostrophic model, see~\cite{CV}, and in the 
free boundary analysis, see~\cite{CSS}.

We notice that~$div\left(y^{\alpha}\nabla\right)$ is an elliptic degenerate operator, 
but, thanks to the fact that~$\alpha\in(-1,1)$, we have that 
the weight~$y^{\alpha}$ is integrable at~$0$. 
This type of weights falls into the set of the so-called~$A_2$-Muckenhoupt weights, see for example~\cite{Mu}. 
Remarkably, an almost complete theory for these equations is available, see~\cite{FJK, FKS}. 
In particular, one can obtain H\"older regularity, 
Poincar\'e-Sobolev-type estimates, Harnack and boundary Harnack principles. 

\bigskip

In this paper we want to show symmetry properties 
for phase separations driven by fractional Laplacian.  
Our proof will be rather simple, and we will require minimal assumptions; 
for instance, we will take the nonlinearities~$F_1,F_2$ to be just locally Lipschitz. 
Our approach  is based on a Poincar\'e-type formula, which involves the tangential gradients 
and the curvatures of the level sets of the solution. This type of inequalities are well known in the case of one equation. They were firstly studied in \cite{SZ1,SZ2} for the classical uniformly elliptic semilinear framework and then they were successfully applied to more general families (also degenerate) of equations in \cite{FSV}. We also recall \cite{FG,Di,DiPi} where Poincar\'e inequality has been studied in the case of systems.

\bigskip

We borrow a large number of ideas from~\cite{FSV} and~\cite{SV}, 
where the authors give some geometric insight on more general types of boundary reactions.

\bigskip 
We consider~$F\in C^{1,1}_{loc}\left(\R^2\right)$, 
and we study the following elliptic system in~$\R^n$
\begin{eqnarray}\label{systfrac}
\left\{ 
\begin{array}{ll} 
    \left(-\Delta\right)^{s_1} u = F_1(u, v),         \\
    \left(-\Delta\right)^{s_2} v= F_2(u, v),
 \end{array} 
\right. 
\end{eqnarray} 
where~$F_1$ and~$F_2$ denote the derivatives of~$F$ with respect to the first and the second variable respectively, and~$s_1, s_2\in (0,1)$. 
 
It is possible to prove, using a Poisson kernel extension (see~\cite{SV}), 
that~\eqref{systfrac} can be reduced 
to the following extension problem
\begin{eqnarray}\label{systext}
\left\{ 
\begin{array}{llll} 
div\left( y^{\alpha_1} \nabla U\right)=0, \qquad & \mathrm{on\ } \R^{n+1}_+ ,  \\
\lim_{y\rightarrow0^+}\left(-y^{\alpha_1} \partial_y U \right)=F_1(U,V), \qquad & \mathrm{on\ }  \partial\R^{n+1}_+, \\
div\left(y^{\alpha_2} \nabla V\right)=0, \qquad & \mathrm{on\ } \R^{n+1}_+,  \\
\lim_{y\rightarrow0^+}\left(-y^{\alpha_2} \partial_y V\right) =F_2(U,V), \qquad & \mathrm{on\ } \partial\R^{n+1}_+,
 \end{array} 
\right.  
\end{eqnarray}
where~$\alpha_1=1-2s_1$ and ~$\alpha_2=1-2s_2$. 

In our setting, we deal with weak solutions to~\eqref{systext}, 
that is, we require that~$U,V\in L^{\infty}_{loc}(\overline{\R^{n+1}_+})$, with 
\begin{equation}\label{cond}
y^{\alpha_1}|\nabla U|^2,\, y^{\alpha_2}|\nabla V|^2 \in L^1 \left(B_R^+\right) 
\end{equation}
for any~$R>0$, and that $U,V$ satisfy\footnote{We notice that~\eqref{weak} 
makes sense, thanks to condition~\eqref{cond}. 
In Lemma~\ref{lem3} we will see that 
if~$U$ and~$V$ are bounded, then it is always satisfied.} 
\begin{eqnarray}\label{weak}
\left\{ 
\begin{array}{ll} 
\displaystyle\int_{\R^{n+1}_+} y^{\alpha_1} \nabla U\cdot\nabla \xi_1 = \displaystyle\int_{\partial\R^{n+1}_+} F_1(U,V)\, \xi_1, \\
\displaystyle\int_{\R^{n+1}_+} y^{\alpha_2} \nabla V\cdot\nabla \xi_2 = \displaystyle\int_{\partial\R^{n+1}_+} F_2(U,V)\, \xi_2,
 \end{array} 
\right. 
\end{eqnarray} 
for any~$\xi_1,\xi_2:B_R^+\rightarrow\R$ bounded, locally Lipschitz 
in the interior of~$\R^{n+1}_+$, which vanish on~$\R^{n+1}_+\setminus B_R$ 
and such that 
\begin{equation}\label{xi} 
y^{\alpha_1}|\nabla\xi_1|^2,\, y^{\alpha_2}|\nabla\xi_2|^2\in L^1\left(B_R^+\right), 
\end{equation}
where, as usual, we denote by~$B_R^+:=B_R\cap\R^{n+1}_+$, 
and~$B_R$ is the ball of radius~$R$ centered at the origin in~$\R^{n+1}$. 

In order to state our main result, we give the definition of monotone and stable solution. 
\begin{defn}\label{monsol}
We say that a solution~$(U, V)$ of~\eqref{systext} satisfies 
a \emph{monotonicity condition} if 
\begin{equation}\label{monotonicity}
 U_{x_n} >0>V_{x_n}\quad \mbox{in}\ \R^{n+1}_{+}.
\end{equation}
\end{defn}

\begin{defn}\label{defstable}
When~$F\in C^2_{loc}(\R^2)$, we say that a solution~$(U,V)$ of 
\eqref{systext} is \emph{stable} if 
\begin{equation}\begin{split}\label{stable}
&\int_{B_R^+}y^{\alpha_1}|\nabla\xi_1|^2 + \int_{B_R^+}y^{\alpha_2}|\nabla\xi_2|^2 \\ &\qquad \qquad - 
\int_{\partial B_R^+}\left( F_{11}(U,V)\, \xi_1^2 + F_{22}(U,V)\, \xi_2^2 + 
2 F_{12}(U,V)\, \xi_1 \, \xi_2 \right)\geq 0,
\end{split}\end{equation}
for any~$\xi_1, \xi_2$ as above.
\end{defn}

In our general framework, since~$F_1$ and~$F_2$ may not be everywhere differentiable, 
the integral in~\eqref{stable} may not be well defined. 
Therefore it is convenient to introduce the sets 
$$ 
  \mathcal{D} :=  \left\lbrace (t,s)\in\R^2 : F_{11}(t,s),\, F_{12}(t,s),\,
 F_{22}(t,s) \mathrm{\ exist} \right\rbrace ,
$$
and 
$$ 
  \mathcal{N} := \R^2 \setminus\mathcal{D}.
$$
It is known that 
\begin{equation}\label{borel}
\mathrm{the\ set\ } \mathcal{N} \mathrm{\ is\ Borel\ and\ with\ zero\ Lebesgue\ measure\ } 
\end{equation}
(see pages 81--82 in~\cite{EG}).
Moreover, we consider the sets
$$
  \mathcal{N}_{uv} := \left\lbrace x\in\R^n : (u(x), v(x))\in\mathcal{N} \right\rbrace, 
$$
and
$$
  \mathcal{D}_{uv} := \R^n \setminus\mathcal{N}_{uv}. 
$$
Therefore, in our setting, we say that~$(U,V)$ is a stable solution to~\eqref{systext}, 
if
\begin{equation}\begin{split}\label{stable1}
&\int_{B_R^+}y^{\alpha_1}|\nabla\xi_1|^2 + \int_{B_R^+}y^{\alpha_2}|\nabla\xi_2|^2 \\ &\qquad \qquad - 
\int_{\partial B_R^+\cap\mathcal{D}_{uv}}\left( F_{11}(U,V)\, \xi_1^2 + F_{22}(U,V)\, \xi_2^2 + 
2 F_{12}(U,V)\, \xi_1 \, \xi_2 \right)\geq 0,
\end{split}\end{equation}
Of course,~\eqref{stable1} reduces to~\eqref{stable} when~$F$ is in~$C^2_{loc}(\R^2)$. 

\begin{rem} 
The stability condition~\eqref{stable} is usually related to minimization problems. 
In particular, it states that the energy functional associated to the system has positive (formal) second variation (we refer to~\cite{AAC, AC, FSV} for more details). 
It is worth noticing that, under an additional assumption on the sign of~$F_{12}$, 
the notion of monotonocity implies the one of stability (see Proposition~\ref{propMS}).
\end{rem} 

\bigskip

According to~\cite{FSV, SZ1, SZ2}, we introduce the following notation. 
For any fixed~$y>0$ and~$c\in\R$, we define the level sets
\begin{eqnarray*} 
S_U &=& S_{U,y,c}:=\left\lbrace x\in\R^n : U(x,y)=c\right\rbrace, \\
S_V &=& S_{V,y,c}:=\left\lbrace x\in\R^n : V(x,y)=c\right\rbrace.
\end{eqnarray*}
We also define 
\begin{eqnarray*} 
L_U &=& L_{U,y,c}:=\left\lbrace x\in S_U : \nabla_x U(x,y)\neq 0\right\rbrace, \\
L_V &=& L_{V,y,c}:=\left\lbrace x\in S_V : \nabla_x V(x,y)\neq 0\right\rbrace.
\end{eqnarray*}

Moreover, we recall that the tangential gradient~$\nabla_{L_U}$ and~$\nabla_{L_V}$ 
along~$L_U$ and~$L_V$ respectively 
is defined for every point~$x_1\in L_U$ and every point~$x_2\in L_V$, and for any~$G:\R^n\rightarrow\R$ smooth in the vicinity of~$x_1$ and~$x_2$ respectively as
\begin{eqnarray*}
\nabla_{L_U}G(x_1) &:=& \nabla_x G(x_1)-\left(\nabla_x G(x_1)\cdot \frac{\nabla_x U(x_1,y)}{|\nabla_x U(x_1,y)|}\right) \frac{\nabla_xU(x_1,y)}{|\nabla_x U(x_1,y)|}, \\
\nabla_{L_V}G(x_2) &:=& \nabla_x G(x_2)-\left(\nabla_x G(x_2)\cdot \frac{\nabla_x V(x_2,y)}{|\nabla_x V(x_2,y)|}\right) \frac{\nabla_x V(x_2,y)}{|\nabla_x V(x_2,y)|}.
\end{eqnarray*} 
Since~$L_U$ and~$L_V$ are smooth manifolds, we can define the total curvature as 
\begin{eqnarray*}
\mathcal K_U &:=& \sqrt{\sum_{j=1}^{n-1}\left(k_{U,j}(x,y)\right)^2}, \qquad \mathrm{for\ any\ }x\in L_U, \\ 
\mathcal K_V &:=& \sqrt{\sum_{j=1}^{n-1}\left(k_{V,j}(x,y)\right)^2}, \qquad \mathrm{for\ any\ }x\in L_V,
\end{eqnarray*}
where 
\begin{eqnarray*}
&& k_{U,1}(x,y), \ldots, k_{U, n-1}(x,y), \qquad \mathrm{for\ any\ }x\in L_U, \\
&& k_{V,1}(x,y), \ldots, k_{V, n-1}(x,y), \qquad \mathrm{for\ any\ }x\in L_V
\end{eqnarray*}
are the principal curvatures of~$L_U$ and~$L_V$ respectively.  

Finally, we set 
\begin{equation}\begin{split}\label{rU}
\mathcal R^{n+1}_U &:= \left\lbrace (x,y)\in \R^n\times (0,+\infty) \mathrm{\ s.t.\ } \nabla_x U(x,y)\neq 0\right\rbrace, \\
\mathcal R^{n+1}_V &:= \left\lbrace (x,y)\in \R^n\times (0,+\infty) \mathrm{\ s.t.\ } \nabla_x V(x,y)\neq 0\right\rbrace.
\end{split}\end{equation}

Now, we state a geometric formula both for monotone and for stable solutions to~\eqref{systext}: 
\begin{theorem}\label{TMon}
Let~$(U,V)$ be a monotone weak solution of~\eqref{systext} such that given $R>0$ there exists $C>0$, depending on $R$, such that
\begin{align}
\|\nabla_x U \|_{L^{\infty}(\R^n\times (0,R))}+ \|\nabla_x V\|_{L^{\infty}(\R^n\times (0,R))}\leq C.
\end{align}

%and suppose that~$U,V$ are~$C^2_{loc}$ in the interior of~$\R^{n+1}_+$.

Then, 
\begin{eqnarray*}
&&\int_{\mathcal R^{n+1}_U} y^{\alpha_1}\left(\mathcal K_U^2 |\nabla_x U|^2  +\Big|\nabla_{L_U}|\nabla_x U|\Big|^2\right) \varphi^2 \\ &&\qquad + \int_{\mathcal R^{n+1}_V} y^{\alpha_2}\left(\mathcal K_V^2 |\nabla_x V|^2 +\Big|\nabla_{L_V}|\nabla_x V|\Big|^2\right) \varphi^2  \\ 
\leq && \int_{\R^{n+1}_+} \left(y^{\alpha_1}|\nabla_xU|^2 +y^{\alpha_2}|\nabla_x V|^2\right) |\nabla\varphi|^2 \\&&\qquad 
+ \int_{\partial\R^{n+1}_+} F_{12}(U,V) \left|\sqrt{\frac{-V_{x_n}}{U_{x_n}}}\nabla_x U +\sqrt{\frac{U_{x_n}}{-V_{x_n}}}\nabla_x V \right|^2\varphi^2,
\end{eqnarray*}
for any~$R>0$, and any Lipschitz function~$\varphi:\R^{n+1}\rightarrow\R$ which vanishes on~$\R^{n+1}_+\setminus B_R$.
\end{theorem}

\begin{theorem}\label{T1}
Let~$(U,V)$ be a stable weak solution of~\eqref{systext} such that given $R>0$ there exists $C>0$, depending on $R$, such that
\begin{equation}\label{estnuova}
\|\nabla_x U \|_{L^{\infty}(\R^n\times (0,R))}+ \|\nabla_x V\|_{L^{\infty}(\R^n\times (0,R))}\leq C.
\end{equation}
%and suppose that~$U,V$ are~$C^2_{loc}$ in the interior of~$\R^{n+1}_+$.

Then, 
\begin{eqnarray*}
&&\int_{\mathcal R^{n+1}_U} y^{\alpha_1}\left(\mathcal K_U^2 |\nabla_x U|^2  +\Big|\nabla_{L_U}|\nabla_x U|\Big|^2\right) \varphi^2 \\ &&\qquad + \int_{\mathcal R^{n+1}_V} y^{\alpha_2}\left(\mathcal K_V^2 |\nabla_x V|^2 +\Big|\nabla_{L_V}|\nabla_x V|\Big|^2\right) \varphi^2  \\ 
\leq && \int_{\R^{n+1}_+} \left(y^{\alpha_1}|\nabla_xU|^2 +y^{\alpha_2}|\nabla_x V|^2\right) |\nabla\varphi|^2 \\&&\qquad - 2\int_{\partial\R^{n+1}_+} F_{12}(U,V) \left(|\nabla_x U|\cdot |\nabla_x V|-\nabla_x U\cdot \nabla_x V\right)\varphi^2,
\end{eqnarray*}
for any~$R>0$, and any Lipschitz function~$\varphi:\R^{n+1}\rightarrow\R$ which vanishes on~$\R^{n+1}_+\setminus B_R$.
\end{theorem}

An immediate consequence of Theorems~\ref{TMon} and~\ref{T1} is the following:
\begin{corollary}\label{cor1}
Let~$(U,V)$ be a weak solution of~\eqref{systext} such that given $R>0$ there exists $C>0$, depending on $R$, such that
\begin{align}
\|\nabla_x U \|_{L^{\infty}(\R^n\times (0,R))}+ \|\nabla_x V\|_{L^{\infty}(\R^n\times (0,R))}\leq C.
\end{align}
%and assume that $U,V$ are~$C^2_{loc}$ in the interior of~$\R^{n+1}_+$. 

Suppose that either
\begin{equation*}
\mbox{the\ monotonicity\ condition\ \eqref{monotonicity}\ holds,\ and\ $F_{12}(U,V)\leq 0$}  , \end{equation*}
or   
\begin{equation*}
  \mbox{$(U,V)$\ is\ stable,\ and\ $F_{12}(U,V)\geq 0$} . 
\end{equation*}  

Then,
\begin{equation}\begin{split}\label{geom}
&\int_{\mathcal R^{n+1}_U} y^{\alpha_1}\left(\mathcal K_U^2 |\nabla_x U|^2  +\Big|\nabla_{L_U}|\nabla_x U|\Big|^2\right) \varphi^2 \\ &\qquad + \int_{\mathcal R^{n+1}_V} y^{\alpha_2}\left(\mathcal K_V^2 |\nabla_x V|^2 +\Big|\nabla_{L_V}|\nabla_x V|\Big|^2\right) \varphi^2  \\ 
\leq & \int_{\R^{n+1}_+} \left(y^{\alpha_1}|\nabla_xU|^2 +y^{\alpha_2}|\nabla_x V|^2\right) |\nabla\varphi|^2,
\end{split}\end{equation}
for any~$R>0$, and any Lipschitz function~$\varphi:\R^{n+1}\rightarrow\R$ which vanishes on~$\R^{n+1}_+\setminus B_R$.
\end{corollary}

Notice that we can consider the geometric formula \eqref{geom} 
as a weighted Poincar\'e inequality, because the weighted 
$L^2$-norm of any test function is bounded by a weighted $L^2$-norm of its gradient.

The second result that we state is a symmetry result in dimension $n=2$ for the system~\eqref{systext}: 
\begin{theorem}\label{T2} 
Let~$(U,V)$ be a bounded weak solution of~\eqref{systext} and let~$n=2$. 
%and assume that $U,V$ are~$C^2_{loc}$ in the interior of~$\R^{n+1}_+$. 

Suppose that either
\begin{equation*}
\mbox{the\ monotonicity\ condition\ \eqref{monotonicity}\ holds,\ and\ $F_{12}(U,V)\leq 0$}  , \end{equation*}
or   
\begin{equation*}
  \mbox{$(U,V)$\ is\ stable,\ and\ $F_{12}(U,V)\geq 0$} . 
\end{equation*}  

Then, there exist~$\omega_U,\omega_V\in S^1$, 
and~$U_0,V_0:\R\times[0,+\infty)\rightarrow\R$ such that 
$$ U(x,y)=U_0(\omega_U\cdot x,y), \qquad V(x,y)=V_0(\omega_V\cdot x,y) $$
for any~$(x,y)\in\R^3_+$. 
\end{theorem}

Theorem~\ref{T2} says that, for any fixed~$y>0$, 
the functions~$x\in\R^2\rightarrow U(x,y)$ and~$x\in\R^2\rightarrow V(x,y)$ 
depend only on one variable.

We finally state the symmetry result for the system~\eqref{systfrac}.
For this, we denote by~$\Im(u,v)$ the image on the
map~$(u,v):\R^n\rightarrow\R^2$, i.e.
$$ \Im(u,v):=\{(u(x),v(x)),\ x\in\R^n\}.$$
Of course, the behavior of~$F$ is relevant for our problem
only at points of~$\Im(u,v)$. Then the following symmetry result
holds:
 
\begin{theorem}\label{T3} 
Let~$u,v\in C^2_{loc}(\R^n)$ be a bounded solution of~\eqref{systfrac}, 
with~$n=2$ and~$F\in C^{1,1}_{loc}(\R^2)$. 

Suppose that either  
$$ u_{x_2}>0>v_{x_2} \mbox{\ and\ $F_{12}(u,v)\leq0$}\quad \mbox{in}\ \Im(u,v), $$
or 
$$ \mbox{condition\ \eqref{stable1}\ holds\ and\ $F_{12}(u,v)\geq0$}\quad \mbox{in}\ \Im(u,v). $$

Then, there exist~$\omega_u,\omega_v\in S^1$ and~$u_0,v_0:\R\rightarrow\R$ such that 
$$ u(x)=u_0(\omega_u\cdot x), \qquad v(x)=v_0(\omega_v\cdot x), $$
for any~$x\in\R^2$. 

Moreover, if we assume in addition that either  
\begin{equation}\begin{split}\label{monF12}
&\mbox{$u_{x_2}>0>v_{x_2}$ \ and\ there\ exists\ a\ non-empty\ open\ set~$\Omega\subseteq\R^2$}\\ &\mbox{such\ that\ $F_{12}(u(x),v(x))<0$\ for\ any\ $x\in\Omega$,}
\end{split}\end{equation}
or
\begin{equation}\label{stabF12}\begin{split}
&\mbox{condition\ \eqref{stable1}\ holds\ and}\\
&\mbox{
there exist two open intervals $I_u$, $I_v\subset\R$ such that
$(I_u\times I_v)\cap \Im(u,v)\ne\varnothing$} \\
&\mbox{
and
$F_{12}(\bar u,\bar v)>0$
for any $(\bar u,\bar v)\in I_u\times I_v$}, 
\end{split}\end{equation}
then~$u$ and~$v$ have one-dimensional symmetry, and~$\omega_{u}=\omega_{v}$\footnote{Notice that if~$F_{12}$ is continuous in~$\R^n$, then both in~\eqref{monF12} and 
in~\eqref{stabF12} it is sufficient to require that there exists a point~$\bar x\in\R^n$ 
such that~$F_{12}(u(\bar x),v(\bar x))<0$ and~$F_{12}(u(\bar x),v(\bar x))>0$ respectively.}.
\end{theorem}

Recently, the preprint~\cite{TVZ} considered the particular case 
of the square root of the Laplacian for the specific potential~$F(u,v)=-u^2v^2$, 
showing, among other things, that solutions with some growth at infinity 
(in particular, bounded solutions) are necessarily constant. 

Our results apply to a more general setting, in which, in general, it is not true that 
bounded solutions are constant, even if they depend only on one variable. 
For instance, our results comprise, as a particular case, 
the uncoupled system of fractional phase transition problems of Allen-Cahn type 
(see~\cite{CSir2, SV}), which possesses heteroclinic solutions. 

On the other hand, these methods may be used, in some circumstances, under some energy 
growth assumptions or some control of the geometric features of the ambient space, 
to prove that a special class of solutions reduces to the constants, 
see~\cite{FSirV, FMV, FV}.

\section{Regularity theory for the systems \eqref{systfrac} and \eqref{systext}}\label{sec:reg}
In this section we prove some regularity results that we will need in the sequel. 
We borrow some ideas from \cite{CSir}.

\begin{lemma}\label{lem:reg}
Let~$(u,v)$ be a bounded weak solution of~\eqref{systfrac} 
and assume that~$F\in C^{1,1}_{loc}(\R^2)$. 
Then~$u\in C^{1,\beta_1}(\R^n)$ and~$v\in C^{1,\beta_2}(\R^n)$, 
for some~$0<\beta_1,\beta_2<1$ depending on~$n,s_1,s_2$ (possibly equal). 
\end{lemma}

\begin{proof} 
Suppose that~$s_1\leq s_2$. 
Since~$u$ and~$v$ are bounded and~$F_1, F_2$ are~$C^{0,1}_{loc}(\R^2)$, $F_1(u,v)$ 
and~$F_2(u,v)$ are also bounded. 
Therefore, we can apply Proposition 2.9 in~\cite{Si} to obtain that 
\begin{itemize}
\item[1)] if~$2s_1>1$, then~$u\in C^{1,\alpha}(\R^n)$ for any~$\alpha<2s_1-1$,
\item[2)] if~$2s_1\leq 1$, then~$u\in C^{0,\alpha}(\R^n)$ for any~$\alpha<2s_1$,
\end{itemize}
and 
\begin{itemize}
\item[1)'] if~$2s_2>1$, then~$v\in C^{1,\alpha}(\R^n)$ for any~$\alpha<2s_2-1$,
\item[2)'] if~$2s_2\leq 1$, then~$v\in C^{0,\alpha}(\R^n)$ for any~$\alpha<2s_2$.
\end{itemize} 

Hence, if~$1)$ and~$1)'$ hold, we have the thesis. 

If~$2)$ and~$1)'$ (respectively~$2)'$) hold, then, for any~$\alpha<\min\left\lbrace 2s_1, 2s_2-1\right\rbrace$ (respectively~$\alpha<2s_1$),~$u\in C^{0, \alpha}(\R^n)$  and~$v\in C^{1,\alpha}(\R^n)$ 
(respectively~$v\in C^{0,\alpha}(\R^n)$). 
Therefore, also~$F_1(u,v)$ and~$F_2(u,v)$ are in~$C^{0,\alpha}(\R^n)$. 

Now, we can apply Proposition 2.8 in~\cite{Si} to obtain that 
\begin{itemize}
\item[i)] if~$\alpha+2s_1>1$, then~$u\in C^{1,\alpha+2s_1-1}(\R^n)$,
\item[ii)] if~$\alpha+2s_1\leq 1$, then~$u\in C^{0,\alpha+2s_1}(\R^n)$,
\end{itemize}
and 
\begin{itemize}
\item[i)'] if~$\alpha+2s_2>1$, then~$v\in C^{1,\alpha+2s_2-1}(\R^n)$,
\item[ii)'] if~$\alpha+2s_2\leq 1$, then~$v\in C^{0,\alpha+2s_2}(\R^n)$.
\end{itemize} 

Hence, if~$i)$ and~$i)'$ hold, we have the thesis. 

Whereas, if~$ii)$ and~$i)'$ (respectively~$ii)'$) hold, 
then~$u\in C^{0,\alpha+2s_1}(\R^n)$ 
and~$v\in C^{1,\alpha+2s_2-1}(\R^n)\subset C^{0,\alpha+2s_2-1}(\R^n)$ 
(respectively~$v\in C^{0,\alpha+2s_2}(\R^n)$), 
which implies that~$u,v\in C^{0,\beta}(\R^n)$, 
with~$\beta=\min\left\lbrace \alpha+2s_1, \alpha+2s_2-1\right\rbrace$ 
(respectively~$\beta=\alpha+2s_1$). 
Indeed, suppose for instance that~$\min\left\lbrace \alpha+2s_1, \alpha+2s_2-1\right\rbrace=\alpha+2s_1$, 
then~$v\in C^{0, \alpha+2s_1}_{loc}(\R^n)$ from the classical embedding of H\"older spaces, which actually means that~$v\in C^{0, \alpha+2s_1}(\R^n)$ since~$v$ is bounded; 
in the same way one proves also the other cases. 

Therefore,~$F_1(u,v)$ and~$F_2(u,v)$ are~$C^{0,\beta}(\R^n)$, 
and we can apply again Proposition 2.8 in~\cite{Si}.

Hence, in a finite number of steps, we will end up with~$\alpha+2ks_1>1$ or~$\alpha+k(2s_2-1)>1$ or~$\alpha+k(2s_1+2s_2-1)>1$ for some integer~$k$. 
This gives the thesis. 
\end{proof}

Now, we recall the following result from~\cite{CS}:
\begin{lemma}\label{lemP}
The function 
\begin{equation}\label{poisson}
P_{\alpha}(x,y) = C_{n,\alpha}\, \frac{y^{1-\alpha}}{(|x|^2+y^2)^{\frac{n+1-\alpha}{2}}}, \quad (x,y)\in \R^n\times (0,+\infty)
\end{equation} 
is a solution of 
\begin{eqnarray}\label{eqpoisson}
\left\{ 
\begin{array}{ll} 
-div\left( y^{\alpha} \nabla P_{\alpha}\right)=0 \qquad & \mathrm{on\ } \R^n\times(0,+\infty),  \\
P_{\alpha}=\delta_0  \qquad & \mathrm{on\ }  \R^n\times\left\lbrace 0\right\rbrace , 
 \end{array} 
\right.         
\end{eqnarray}
where~$\alpha\in (-1,1)$ and~$C_{n,\alpha}$ is a normalizing constant such that 
$$ \int_{\R^n} P_{\alpha}(x,y) \, dx=1. $$ 
\end{lemma}

Now, if~$u,v$ are bounded solutions to~\eqref{systfrac}, 
we consider the functions
\begin{equation}\label{extU}
U(x,y)=\int_{\R^n}P_{\alpha_1}(x-z,y)u(z)\, dz = \int_{\R^n}P_{\alpha_1}(\zeta,y)u(x-\zeta)\, d\zeta,
\end{equation}
\begin{equation}\label{extV}
V(x,y)=\int_{\R^n}P_{\alpha_2}(x-z,y)v(z)\, dz = \int_{\R^n}P_{\alpha_2}(\zeta,y)v(x-\zeta)\, d\zeta,
\end{equation}
where~$\alpha_1=1-2s_1$ and~$\alpha_2=1-2s_2$. 
We observe that~$U,V$ 
are bounded in~$\R^{n+1}_+$, because~$u,v$ are bounded in~$\R^n$  and~$P_{\alpha_i}(x,\cdot)\in L^1(\R^n)$ for~$i=1,2$. 

\begin{lemma}\label{lem:regUV}
Let~$(u,v)$ be a bounded weak solution of~\eqref{systfrac} 
and assume that~$F\in C^{1,1}_{loc}(\R^2)$. Let~$U,V$ be the functions defined in~\eqref{extU} and~\eqref{extV} respectively, then 
$$ \|U\|_{C^{0,\beta_1}(\overline{\R^{n+1}_+})} + \|\nabla_x U\|_{C^{0,\beta_1}(\overline{\R^{n+1}_+})} \leq C_1 $$
and 
$$ \|V\|_{C^{0,\beta_2}(\overline{\R^{n+1}_+})} + \|\nabla_x V\|_{C^{0,\beta_2}(\overline{\R^{n+1}_+})} \leq C_2 $$
for some positive constants 
$$C_1=C_1(n,s_1,s_2,\|F_1\|_{C^{0,1}(\R^2)},\|u\|_{L^{\infty}(\R^n)})$$ 
and 
$$C_2=C_2(n,s_1,s_2,\|F_2\|_{C^{0,1}(\R^2)},\|v\|_{L^{\infty}(\R^n)}).$$ 
\end{lemma}

\begin{proof} 
We first notice that we can rewrite the Poisson kernel in~\eqref{poisson} as 
\begin{eqnarray*}
P_{\alpha}(x,y) &=& C_{n,\alpha}\frac{y^{1-\alpha}}{\left(|x|^2+y^2\right)^{\frac{n+1-\alpha}{2}}} \\ 
&=& C_{n,\alpha} \frac{y^{1-\alpha}}{y^{n+1-\alpha}\left(\frac{|x|^2}{y^2}+1\right)^{\frac{n+1-\alpha}{2}}} \\ &=& C_{n,\alpha} \frac{y^{-n}}{\left(\frac{|x|^2}{y^2}+1\right)^{\frac{n+1-\alpha}{2}}} \\ &=& C_{n,\alpha} \frac{1}{y^n} Q_{\alpha}\left(\frac{x}{y}\right),
\end{eqnarray*}
where 
$$ Q_{\alpha}(z):= \frac{1}{\left(|z|^2+1\right)^{\frac{n+1-\alpha}{2}}}. $$

Hence, we can write~$U$ as 
\begin{eqnarray*}
U(x,y) &=& \int_{\R^n} P_{\alpha_1}(\zeta,y)u(x-\zeta)\, d\zeta \\
&=& C_{n,\alpha_1}\int_{\R^n}\frac{1}{y^n}Q_{\alpha_1}\left(\frac{\zeta}{y}\right) u(x-\zeta)\, d\zeta \\ &=& C_{n,\alpha_1} \int_{\R^n}Q_{\alpha_1}(z)u(x-yz)\, dz,
\end{eqnarray*}
by the change of variable~$\frac{\zeta}{y}=z$. 

Hence, 
\begin{eqnarray*}
&& |U(x_1,y_1)-U(x_2,y_2)|
\\ &\leq& C_{n,\alpha_1}\int_{\R^n}|u(x_1-y_1z)-u(x_2-y_2z)| 
Q_{\alpha_1}(z) \, dz \\
&\leq& C\|u\|_{L^{\infty}(\R^n)}\int_{\R^n}|x_1-y_1z-x_2-y_2z|^{\beta_1}Q_{\alpha_1}(z)\, dz \\ &\leq& C\|u\|_{L^{\infty}(\R^n)}\left(|x_1-x_2|^{\beta_1}+ |y_1-y_2|^{\beta_1}\int_{\R^n} |z|^{\beta_1}Q_{\alpha_1}(z)\, dz\right).
\end{eqnarray*}
Applying this fact also to~$U_{x_j}$ for any~$j=1,\ldots,n$ we get the conclusion for~$U$. 

In the same way we obtain the estimates for~$V$ and this concludes the proof. 
\end{proof}

The following result concerns a bound for solutions of problem~\eqref{systext}. 
\begin{prop}\label{prop:reg}
Let~$(u,v)$ be a bounded weak solution of~\eqref{systfrac} 
and assume that $F\in C^{1,1}_{loc}(\R^2)$. Let~$U,V$ be the functions defined 
in~\eqref{extU} and~\eqref{extV}. 

Then, given~$R>0$, there exists~$C>0$, depending on $R$, such that 
\begin{equation}\label{gradLinf} 
\|\nabla_x U\|_{L^{\infty}(\R^n\times(0,R))} +\|\nabla_x V\|_{L^{\infty}(\R^n\times(0,R))} \leq C. 
\end{equation}
\end{prop}

\begin{proof}
We notice that, thanks to Lemma~\ref{lem:regUV}, $\nabla_x U$ and~$\nabla_x V$ are bounded, 
for instance, in~$\overline{\R^{n+1}_+}\cap\left\lbrace 0\leq y\leq1\right\rbrace$. 
Whereas, in~$\overline{\R^{n+1}_+}\cap\left\lbrace y>1\right\rbrace $, 
the equations in~\eqref{systext} are 
nondegenerate and therefore standard elliptic arguments imply the gradient bounds.
\end{proof}

\section{Some useful Lemmas}
In this section we prove some lemmas that will be useful in the sequel. 

First of all, we obtain energy estimates needed for the proof of Theorem~\ref{T2}.
\begin{lemma}\label{lem3} 
Let~$(U,V)$ be a weak solution of~\eqref{systext}.

Then, for any~$R>0$ there exists~$C$, possibly depending on~$R$, such that 
\begin{equation}\label{CR}
 \|y^{\alpha_1}|\nabla U|^2\|_{L^1(B_R^+)}\leq C, \qquad 
\|y^{\alpha_2}|\nabla V|^2\|_{L^1(B_R^+)}\leq C. 
\end{equation}

Moreover, if~$n=2$, then there exists~$C_0>0$ such that for every~$R\geq 1$
\begin{equation}\label{condy}
\int_{B_R^+}y^{\alpha_1}|\nabla U|^2\leq C_0 R^2, \qquad \int_{B_R^+}y^{\alpha_2}|\nabla V|^2\leq C_0 R^2. 
\end{equation}
\end{lemma}

\begin{proof} 
We choose a cutoff function~$\varphi\in C^{\infty}_0(B^+_{2R})$, with~$\varphi\geq0$, $\varphi=1$ in~$B^+_R$ and~$|\nabla\varphi|\leq\frac{C}{R}$, 
with~$R\geq1$, 
and we test the weak formulation in~\eqref{weak} with~$\xi_1=U\varphi^2$ and~$\xi_2=V\varphi^2$. 
We notice that, thanks to~\eqref{cond} and the properties of~$\varphi$, 
the condition~\eqref{xi} is satisfied, and then these test functions are admissible. 

Then, from the first equation in~\eqref{weak}, one gets 
$$
\int_{\R^{n+1}_+}y^{\alpha_1}\left(|\nabla U|^2\varphi^2 +2\varphi\nabla U\cdot\nabla\varphi\right) =\int_{\R^n}F_{1}(U,V)U\varphi^2.
$$
Therefore, using the Cauchy-Schwarz inequality, we have
\begin{eqnarray*}
\int_{\R^{n+1}_+}y^{\alpha_1}|\nabla U|^2\varphi^2 &=&  -2\int_{\R^{n+1}_+}y^{\alpha_1}\varphi\nabla U\cdot\nabla\varphi +\int_{\R^n}F_{1}(U,V)U\varphi^2\\ 
&\leq& \frac{1}{2}\int_{\R^{n+1}_+}y^{\alpha_1}|\nabla U|^2\varphi^2 + 2\int_{\R^{n+1}_+}y^{\alpha_1}|\nabla\varphi|^2\\ 
&&\quad + \int_{\R^n} |F_1(U,V)|\, |U| \varphi^2.
\end{eqnarray*}
This and the properties of the cutoff function~$\varphi$ imply that 
\begin{eqnarray*}
\int_{B_R^+}y^{\alpha_1}|\nabla U|^2 &\leq& \frac{C}{R^2}\int_{B_{2R}^+\setminus B_R^+}y^{\alpha_1} + 2\int_{\left\lbrace |x|\leq 2R\right\rbrace \cap\left\lbrace y=0\right\rbrace}|F_1(U,V)| |U|\\
&\leq&\frac{C}{R^2} \int_0^{2R}\int_{\left\lbrace |x|\leq 2R\right\rbrace \cap\left\lbrace y=0\right\rbrace}y^{\alpha_1}\, dx\, dy +C R^n \\
&\leq& CR^{n-1+\alpha_1} +CR^n,
\end{eqnarray*}
which gives~\eqref{CR}, 
and in particular, if~$n=2$ and~$R\geq 1$,~\eqref{condy}.

In the same way we obtain the same estimates for~$V$. 
\end{proof}

Next we obtain a bound for further derivatives in~$x$:
\begin{lemma}\label{lem2} 
Let~$(U,V)$ be a weak solution of~\eqref{systext}.
Suppose that~\eqref{gradLinf} holds. Then, for any~$r>0$, we have that  
$$ y^{\alpha_1}|\nabla U_{x_j}|^2, \, y^{\alpha_2}|\nabla V_{x_j}|^2 \in L^1 (B_r^+). $$
\end{lemma}

\begin{proof} 
Given~$|\eta|<1$,~$\eta\neq0$, we consider the incremental quotient of~$U$ and~$V$
$$ U_{\eta}(x,y):=\frac{U(x_1,\ldots,x_j+\eta,\ldots,x_n,y)-U(x_1,\ldots,x_j,\ldots,x_n,y)}{\eta}, $$
$$ V_{\eta}(x,y):=\frac{V(x_1,\ldots,x_j+\eta,\ldots,x_n,y)-V(x_1,\ldots,x_j,\ldots,x_n,y)}{\eta}. $$

Since~$F_1$ is locally Lipschitz and~\eqref{gradLinf} holds, we have
\begin{eqnarray}\label{51} 
[F_1(U,V)]_{\eta} &=& \frac{1}{\eta}\Big(F_1(U(x_1,\ldots,x_j+\eta,\ldots,x_n,0), V(x_1,\ldots,x_j+\eta,\ldots,x_n,0)) \nonumber\\ 
&&\qquad -F_1(U(x_1,\ldots,x_j,\ldots,x_n,0), V(x_1,\ldots,x_j,\ldots,x_n,0))\Big)  \nonumber\\
&\leq& \frac{C}{\eta}\Big(|U(x_1,\ldots,x_j+\eta,\ldots,x_n,0)-U(x_1,\ldots,x_j,\ldots,x_n,0)| \nonumber\\ &&\qquad + |V(x_1,\ldots,x_j+\eta,\ldots,x_n,0)-V(x_1,\ldots,x_j,\ldots,x_n,0)|\Big) \nonumber\\
&\leq& C\left(|U_{\eta}(x,0)|+|V_{\eta}(x,0)|\right) \nonumber\\
&\leq& C\left(|\nabla_x U(x,0)|+|\nabla_x V(x,0)|\right) \nonumber\\
&\leq& C. 
\end{eqnarray}

Now, we take~$\xi_1$ satisfying the conditions required in~\eqref{weak}. 
Then, the first equation in~\eqref{weak} implies that 
\begin{equation}\begin{split}\label{aggiunta} 
&\int_{\R^{n+1}_+}y^{\alpha_1}\nabla U_{\eta}\cdot\nabla\xi_1 -\int_{\R^n}[F_1(U,V)]_{\eta} \, \xi_1  \\
&= -\int_{\R^{n+1}_+}y^{\alpha_1}\nabla U\cdot\nabla(\xi_1)_{-\eta} +\int_{\R^n}F_1(U,V)(\xi_1)_{-\eta}  \\
&= 0.
\end{split}\end{equation}

We consider a smooth cutoff function~$\varphi\in C^{\infty}_0(B_{R+1})$ 
such that~$\varphi\geq0$, $\varphi=1$ in~$B_R$ and~$|\nabla\varphi|\leq2$. 
Then, taking~$\xi_1:=U_{\eta}\varphi^2$ in~\eqref{aggiunta}, one obtains
\begin{equation}\label{55} 
\int_{\R^{n+1}_+}y^{\alpha_1}\left(\varphi^2|\nabla U_{\eta}|^2 + 2\varphi\, U_{\eta}\nabla U_{\eta}\cdot\nabla\varphi\right) 
 = \int_{\R^n}[F_1(U,V)]_{\eta} \, U_{\eta}\, \varphi^2.
\end{equation}
We notice that, thanks to~\eqref{cond}, \eqref{gradLinf} and the properties of the function~$\varphi$, the conditions in~\eqref{xi} are satisfied, 
and therefore the above choice of~$\xi_1$ is admissible. 

Now, using the Cauchy-Schwarz inequality, we have that, for any~$\epsilon>0$,
$$ \int_{\R^{n+1}_+}y^{\alpha_1}\varphi\, U_{\eta}\nabla U_{\eta}\cdot\nabla\varphi \geq 
-\frac{\epsilon}{2}\int_{\R^{n+1}_+}y^{\alpha_1}\varphi^2 |\nabla U_{\eta}|^2 
- \frac{1}{2\epsilon}\int_{\R^{n+1}_+}y^{\alpha_1}U_{\eta}^2|\nabla\varphi|^2. $$

Hence, for~$\epsilon$ sufficiently small, \eqref{55} gives
$$ 
\int_{\R^{n+1}_+}y^{\alpha_1}\varphi^2|\nabla U_{\eta}|^2 \leq C\left(\int_{B_{R+1}^+}y^{\alpha_1} U_{\eta}^2 + \int_{\left\lbrace |x|\leq R+1\right\rbrace \times\left\lbrace y=0\right\rbrace} [F_1(U,V)]_{\eta}\, U_{\eta}\right).
$$

Thus, from~\eqref{gradLinf} and~\eqref{51}, we have that
$$ \int_{B_R^+}y^{\alpha_1}|\nabla U_{\eta}|^2\leq C $$
uniformly in~$\eta$. 
Now, we send~$\eta\rightarrow0$ and we use Fatou lemma. 
Then we get the thesis.  

Reasoning in the same way we obtain the same claim for~$V$.
\end{proof}

Next, we state some regularity results that we will need for some subsequent computations 
(see~\cite{SV} for the proof).
\begin{lemma}\label{lem1}  
Let~$(U,V)$ be a weak solution of~\eqref{systext} satisfying \eqref{gradLinf}.
%and suppose that~$U,V$ are~$C^2_{loc}$ in the interior of~$\R^{n+1}_+$.

Then,
\begin{equation}\begin{split}\label{W11}
&\mbox{for\ almost\ any\ } y>0, \mbox{\ the\ maps\ } x\in\R^n\mapsto\nabla U(x,y) \\
&\mbox{and\ } x\in\R^n\mapsto\nabla V(x,y) \mbox{\ are\ in\ } W^{1,1}_{loc}(\R^n,\R^{n+1}),
\end{split}\end{equation}
and 
\begin{equation}\begin{split}\label{L1}
&\mbox{the\ maps\ } (x,y)\in\R^{n+1}_+\mapsto y^{\alpha_1}\sum_{j=1}^n\left(|\nabla U_{x_j}|^2 + |U_{x_j}|^2\right) \\
&\mbox{and\ } (x,y)\in\R^{n+1}_+\mapsto y^{\alpha_2}\sum_{j=1}^n\left(|\nabla V_{x_j}|^2 + |V_{x_j}|^2\right) \\ & \mbox{\ are\ in\ } L^1(B_r^+), \mbox{\ for\ any\ } r>0.
\end{split}\end{equation}
Moreover, 
\begin{equation}\begin{split}\label{L1bis}
&\mbox{the\ maps\ } (x,y)\in\R^{n+1}_+\mapsto y^{\alpha_1}\left(\Big|\nabla |\nabla_x U|\Big|^2 + |\nabla_x U|^2\right) \\
&\mbox{and\ } (x,y)\in\R^{n+1}_+\mapsto y^{\alpha_2}\left(\Big|\nabla |\nabla_x V|\Big|^2 + |\nabla_x V|^2\right) \\ & \mbox{\ are\ in\ } L^1(B_r^+), \mbox{\ for\ any\ } r>0.
\end{split}\end{equation}
\end{lemma}

%\begin{proof} 
%Since~$U,V$ are~$C^2_{loc}$ in the interior of~$\R^{n+1}_+$, 
%for any~$y\in (\epsilon, 1/\epsilon)$ and any~$r>0$, 
%\begin{eqnarray*} 
%&& \int_{B_r}\left( |\nabla U(x,y)| +\sum_{j=1}^n |\nabla U_{x_j}(x,y)|\right) dx \leq C, \\
%&& \int_{B_r}\left( |\nabla V(x,y)| +\sum_{j=1}^n |\nabla V_{x_j}(x,y)|\right) dx \leq C,
%\end{eqnarray*}
%for a suitable positive constant~$C$, possibly depending on~$\epsilon$ and~$r$. 
%This proves~\eqref{W11}.

%Exploiting Lemma~\ref{lem2},~\eqref{gradLinf} and the local integrability 
%of~$y^{\alpha_1}, y^{\alpha_2}$, we obtain also~\eqref{L1}.

%To prove~\eqref{L1bis} we can reason as in the proof of Lemma~$7$ in~\cite{SV}. 
%This concludes the proof of the lemma. 
%\end{proof}

\section{A density result}
In this section we prove a density result that will be useful in the sequel.

We use~\eqref{W11} to say that
\begin{equation}\begin{split}\label{equ}
\int_{\R^{n+1}_+}y^{\alpha_1}\nabla U_{x_j}\cdot\phi = \int_0^{\infty}y^{\alpha_1}\left(\int_{\R^n}\nabla U_{x_j}\cdot\phi\, dx\right) dy = 
-\int_{\R^{n+1}_+} y^{\alpha_1}\nabla U\cdot\phi_{x_j}, \\
\int_{\R^{n+1}_+}y^{\alpha_2}\nabla V_{x_j}\cdot\psi = \int_0^{\infty}y^{\alpha_2}\left(\int_{\R^n}\nabla V_{x_j}\cdot\psi\, dx\right) dy = 
-\int_{\R^{n+1}_+} y^{\alpha_2}\nabla V\cdot\psi_{x_j}, 
\end{split}\end{equation}
for any~$j=1,\ldots,n$ and any~$\phi,\psi\in C^{\infty}(\R^{n+1}_+, \R^n)$. 

Now, using the first equality in~\eqref{equ} and the first equation in~\eqref{weak}, 
we obtain that, for any~$j=1,\ldots,n$ and any~$\phi\in C^{\infty}(\R^{n+1}_+)$ 
supported in~$B_R$,
\begin{equation}\begin{split}\label{estU}
\int_{\R^{n+1}_+} y^{\alpha_1} \nabla U_{x_j}\cdot\nabla\phi &= 
-\int_{\R^{n+1}_+}y^{\alpha_1} \nabla U\cdot\nabla\phi_{x_j} \\
&= -\int_{\partial\R^{n+1}_+} F_1(U,V)\, \phi_{x_j} \\
&= \int_{\mathcal D_{uv}}\left(F_1(U,V)\right)_{x_j} \phi \\
&= \int_{\mathcal D_{uv}} \left( F_{11}(U,V)U_{x_j} + F_{12}(U,V)V_{x_j}\right)\phi.
\end{split}\end{equation}
In the same way, using the second equality in~\eqref{equ} and the second 
equation in~\eqref{weak}, we have that, for any~$j=1,\ldots,n$ and 
any~$\psi\in C^{\infty}(\R^{n+1}_+)$ supported in~$B_R$, 
\begin{equation}\begin{split}\label{estV}
\int_{\R^{n+1}_+} y^{\alpha_2} \nabla V_{x_j}\cdot\nabla\psi &= -\int_{\R^{n+1}_+}y^{\alpha_2} \nabla V\cdot\nabla\psi_{x_j} \\
&= -\int_{\partial\R^{n+1}_+} F_2(U,V)\, \psi_{x_j} \\
&= \int_{\mathcal D_{uv}}\left(F_2(U,V)\right)_{x_j} \psi \\
&= \int_{\mathcal D_{uv}} \left( F_{12}(U,V)U_{x_j} + F_{22}(U,V)V_{x_j}\right)\psi.
\end{split}\end{equation}

In the next sections we will need to use~\eqref{estU} and~\eqref{estV} for 
less regular test functions. 
To do this, we prove the following: 
\begin{lemma}\label{lemdensity}
Let~$(U,V)$ be a weak solution of~\eqref{systext} satisfying \eqref{gradLinf}.
%and suppose that~$U,V$ are~$C^2_{loc}$ in the interior of~$\R^{n+1}_+$. 

Then, we have that~\eqref{estU} and~\eqref{estV} hold for any~$j=1,\ldots,n$, 
any~$\phi,\psi\in W^{1,2}_0(B)$ and any ball~$B\subset\R^{n+1}_+$.  
\end{lemma}

\begin{proof} 
Let us prove~\eqref{estU}. In the same way one can obtain also~\eqref{estV}. 

Given~$\phi\in W^{1,2}_0(B)$, we consider a sequence of functions~$\phi_k\in C^{\infty}_0(B)$ which converge to~$\phi$ in~$W^{1,2}_0(B)$. 
Therefore, since~\eqref{estU} holds 
for any function~$\phi_k\in C^{\infty}(\R^{n+1}_+)$ supported in~$B$, 
we have that 
\begin{equation}\label{equderiv}
\int_{\R^{n+1}_+}y^{\alpha_1}\nabla U_{x_j}\cdot\nabla\phi_k = \int_{\mathcal D_{uv}} \left( F_{11}(U,V)U_{x_j} + F_{12}(U,V)V_{x_j}\right)\phi_k.  
\end{equation}
Also, 
\begin{eqnarray*}
&&\left|\int_{\R^{n+1}_+}y^{\alpha_1}\nabla U_{x_j}\cdot\nabla\left(\phi_k-\phi\right)\right| \\&&\qquad +\left|\int_{\mathcal D_{uv}} \left( F_{11}(U,V)U_{x_j} + F_{12}(U,V)V_{x_j}\right)\left(\phi_k-\phi\right)\right| \\
&\leq& \left(\int_{B}y^{\alpha_1}|\nabla U_{x_j}|^2\right)^{1/2}\left(\int_{B}|\nabla\left(\phi_k-\phi\right)|^2\right)^{1/2}\left(\int_{B}y^{\alpha_1}\right)^{1/2} \\&&\qquad + 
\left(\int_{B\cap\mathcal D_{uv}}|F_{11}(U,V)U_{x_j}+F_{12}(U,V)V_{x_j}|^2\right)^{1/2} \left(\int_{B}|\phi_k-\phi|^2\right)^{1/2}, 
\end{eqnarray*}
which tends to zero as~$k\rightarrow+\infty$, thanks to~\eqref{L1}, the local integrability 
of~$y^{\alpha_1}$ and the assumptions on~$U$. 
The latter consideration and~\eqref{equderiv} give~\eqref{estU}. 
\end{proof}

\section{Monotone solutions and proof of Theorem \ref{TMon}}
Recalling the definition of monotone solution given in~\eqref{monotonicity}, 
in this section we obtain some geometric inequalities and we prove 
Theorem~\ref{TMon}. 

\begin{prop}\label{prop1} 
Let~$(U,V)$ is a weak solution of~\eqref{systext} satisfying \eqref{gradLinf}. 
Suppose that
%~$U,V$ are~$C^2_{loc}$ in the interior of~$\R^{n+1}_+$, and that 
the monotonicity condition~\eqref{monotonicity} holds. 

Then, 
\begin{equation}\begin{split}\label{ineq1}
& \int_{\R^{n+1}_+}y^{\alpha_1}|\nabla\xi_1|^2\geq \int_{\mathcal D_{uv}}\left(F_{11}(U,V)+F_{12}(U,V)\frac{V_{x_n}}{U_{x_n}}\right)\xi_1^2 \\
\mbox{and\ } 
&\int_{\R^{n+1}_+}y^{\alpha_2}|\nabla\xi_2|^2\geq \int_{\mathcal D_{uv}}\left(F_{12}(U,V)\frac{U_{x_n}}{V_{x_n}}+F_{22}(U,V)\right)\xi_2^2,
\end{split}\end{equation}
for any~$\xi_1,\xi_2:B_R^+\rightarrow\R$ bounded, locally Lipschitz 
in the interior of~$\R^{n+1}_+$, which vanish on~$\R^{n+1}_+\setminus B_R$ 
and such that 
$$ y^{\alpha_1}|\nabla\xi_1|^2, y^{\alpha_2}|\nabla\xi_2|^2\in L^1 (B_R^+). $$
\end{prop}

\begin{proof} 
We exploit~\eqref{estU} with~$\phi:=\frac{\xi_1^2}{U_{x_n}}$. 
Notice that Lemma~\ref{lemdensity} implies that we can use it as test function in~\eqref{estU}. Therefore, we have
\begin{eqnarray*}
&&\int_{\mathcal D_{uv}}\left(F_{11}(U,V)+F_{12}(U,V)\frac{V_{x_n}}{U_{x_n}}\right)\xi_1^2 \\
&=& \int_{\R^{n+1}_+}y^{\alpha_1}\nabla U_{x_n}\cdot\nabla\left(\frac{\xi_1^2}{U_{x_n}}\right) \\
&=& \int_{\R^{n+1}_+}y^{\alpha_1}\nabla U_{x_n}\cdot\left(\frac{2\xi_1\nabla\xi_1 U_{x_n}-\xi_1^2\nabla U_{x_n}}{U_{x_n}^2}\right) \\
&=& \int_{\R^{n+1}_+}y^{\alpha_1}\left(2\xi_1\frac{\nabla\xi_1\cdot\nabla U_{x_n}}{U_{x_n}} -\xi_1^2\frac{|\nabla U_{x_n}|^2}{U_{x_n}^2}\right) \\
&=& - \int_{\R^{n+1}_+}y^{\alpha_1}\left|\xi_1\frac{\nabla U_{x_n}}{U_{x_n}}-\nabla\xi_1\right|^2 + \int_{\R^{n+1}_+}y^{\alpha_1}|\nabla\xi_1|^2 \\
&\leq& \int_{\R^{n+1}_+}y^{\alpha_1}|\nabla\xi_1|^2,
\end{eqnarray*}
which proves the first inequality in~\eqref{ineq1}. 
In the same way one can prove also the second inequality and this concludes the proof. 
\end{proof}

\begin{theorem}\label{Tineq}
Let~$U,V$ as in the hypotheses of Theorem~\ref{TMon}. 

Then, we have that 
\begin{equation}\begin{split}\label{ineq2}
& \int_{\mathcal R^{n+1}_U}y^{\alpha_1}\left(\mathcal K_U^2|\nabla_x U|^2 +\Big|\nabla_{L_U}|\nabla_x U|\Big|^2\right)\phi^2 \\
\leq & \int_{\R^{n+1}_+}y^{\alpha_1}|\nabla_x U|^2|\nabla\phi|^2 
+ \int_{\R^n}F_{12}(U,V)\left(\nabla_x U\cdot\nabla_x V-\frac{V_{x_n}}{U_{x_n}}|\nabla_x U|^2\right)\phi^2,
\end{split}\end{equation}
and 
\begin{equation}\begin{split}\label{ineq3}
& \int_{\mathcal R^{n+1}_V}y^{\alpha_2}\left(\mathcal K_V^2|\nabla_x V|^2 +\Big|\nabla_{L_V}|\nabla_x V|\Big|^2\right)\psi^2 \\
\leq & \int_{\R^{n+1}_+}y^{\alpha_2}|\nabla_x V|^2|\nabla\psi|^2 
+ \int_{\R^n}F_{12}(U,V)\left(\nabla_x U\cdot\nabla_x V-\frac{U_{x_n}}{V_{x_n}}|\nabla_x V|^2\right)\psi^2,
\end{split}\end{equation}
for any~$R>0$, and any Lipschitz functions~$\phi,\psi:\R^{n+1}\rightarrow\R$ which 
vanish on~$\R^{n+1}_+\setminus B_R$.
\end{theorem}

\begin{proof}
We prove first~\eqref{ineq2}. By using the first inequality in~\eqref{ineq1} 
with~$\xi_1:=|\nabla_x U|\phi$, we have
\begin{equation}\begin{split}\label{20}
0 &\leq \int_{\R^{n+1}_+}y^{\alpha_1}\Big|\nabla\left(|\nabla_x U|\phi\right)\Big|^2 
 -\int_{\mathcal D_{uv}}\left(F_{11}(U,V)+F_{12}(U,V)\frac{V_{x_n}}{U_{x_n}}\right)\phi^2 |\nabla_x U|^2
\\ &= \int_{\R^{n+1}_+}y^{\alpha_1}\left[\Big|\nabla\left(|\nabla_x U|\right)\Big|^2\phi^2 + |\nabla_x U|^2 |\nabla\phi|^2 +2|\nabla_x U|\phi\nabla\phi\cdot\nabla\left(|\nabla_x U|\right)\right] \\ 
&\qquad -\int_{\mathcal D_{uv}}\left(F_{11}(U,V)|\nabla_x U|^2+F_{12}(U,V)\frac{V_{x_n}}{U_{x_n}}|\nabla_x U|^2\right)\phi^2.
\end{split}\end{equation}

Now, thanks to Lemma~\ref{lemdensity}, we can use~\eqref{estU} with~$U_{x_j}\phi^2$ as test function: 
\begin{eqnarray*}
&& \int_{\mathcal D_{uv}}\left(F_{11}(U,V)\, U_{x_j}^2+ F_{12}(U,V)\, U_{x_j}\, V_{x_j}\right)\phi^2 \\
&=& \int_{\R^{n+1}_+}y^{\alpha_1}\nabla U_{x_j}\cdot\nabla\left(U_{x_j}\phi^2\right) \\ 
&=& \int_{\R^{n+1}_+}y^{\alpha_1}\left(|\nabla U_{x_j}|^2\phi^2+2U_{x_j}\, \phi\nabla\phi\cdot\nabla U_{x_j}\right). 
\end{eqnarray*}
We sum over~$j=1,\ldots,n$,
\begin{equation}\begin{split}\label{21}
&\int_{\mathcal D_{uv}}\left(F_{11}(U,V)|\nabla_x U|^2+F_{12}(U,V)\nabla_x U\cdot\nabla_x V\right)\phi^2 \\ =& \int_{\R^{n+1}_+}y^{\alpha_1}\left[\sum_{j=1}^n|\nabla U_{x_j}|^2\phi^2 + \phi\nabla\phi\cdot\nabla\left(|\nabla_x U|^2\right)\right].  
\end{split}\end{equation}

Putting together~\eqref{20} and~\eqref{21} we get
\begin{eqnarray*}
&&\int_{\R^{n+1}_+} y^{\alpha_1}\left[\sum_{j=1}^n|\nabla U_{x_j}|^2\phi^2 + \phi\nabla\phi\cdot\nabla\left(|\nabla_x U|^2\right)\right] \\ 
&\leq& \int_{\R^{n+1}_+}y^{\alpha_1}\left[\Big|\nabla\left(|\nabla_x U|\right)\Big|^2 \phi^2 + |\nabla_x U|^2|\nabla\phi|^2+\phi\nabla\phi\cdot\nabla\left(|\nabla_x U|^2\right)\right] \\ &&\qquad +\int_{\mathcal D_{uv}}F_{12}(U,V)\left(\nabla_x U\cdot\nabla_x V - \frac{V_{x_n}}{U_{x_n}}|\nabla_x U|^2\right) \phi^2.
\end{eqnarray*}
This implies that 
\begin{equation}\begin{split}\label{22}
&\int_{\R^{n+1}_+}y^{\alpha_1}\left[\sum_{j=1}^n|\nabla U_{x_j}|^2-\Big|\nabla\left(|\nabla_x U|\right)\Big|^2\right]\phi^2 \\ 
\leq& \int_{\R^{n+1}_+}y^{\alpha_1}|\nabla_x U|^2|\nabla\phi|^2 + \int_{\mathcal D_{uv}}F_{12}(U,V)\left(\nabla_x U\cdot\nabla_x V-\frac{V_{x_n}}{U_{x_n}}|\nabla_x U|^2\right)\phi^2.
\end{split}\end{equation}

Now, we take~$r_1,r_2>0$ and~$P\in\R^{n+1}_+$ such that~$B_{r_1+r_2}(P)\subset\R^{n+1}_+$. 
From~\eqref{L1} and~\eqref{L1bis} we have that~$|\nabla_x U|$ and~$U_{x_j}$ are 
in~$W^{1,2}(B_r(P))$, and therefore in $W^{1,1}_{loc}(B_r(P))$. 

Then, by Stampacchia theorem (see, for instance, 
Theorem~$6.19$ in~\cite{LL}), we obtain that
$$\nabla\left(|\nabla_x U|\right)=0$$ 
for almost any~$(x,y)\in B_{r_1}(P)$ such 
that~$|\nabla_x U|=0$, and~$\nabla U_{x_j}=0$ for almost any~$(x,y)\in B_{r_1}(P)$ 
such that~$U_{x_j}=0$. 

Now, since we can take~$P,r_1$ and~$r_2$ arbitrarily, we obtain 
that~$\nabla\left(|\nabla_x U|\right)=0=\nabla U_{x_j}$ for almost 
any~$(x,y)$ such that~$\nabla_x U(x,y)=0$. 

Therefore, we can write \eqref{22} as
\begin{eqnarray*}
&&\int_{\mathcal R^{n+1}_U}y^{\alpha_1}\left(\sum_{j=1}^n\left(\partial_y U_{x_j}\right)^2 - \left(\partial_y|\nabla_x U|\right)^2\right)\phi^2 \\ 
&&\qquad + \int_{\mathcal R^{n+1}_U}y^{\alpha_1}\left(\sum_{j=1}^n|\nabla_x U_{x_j}|^2 - \Big|\nabla_x\left(|\nabla_x U|\right)\Big|^2\right)\phi^2 \\ 
&\leq& \int_{\R^{n+1}_+}y^{\alpha_1}|\nabla_x U|^2|\nabla\phi|^2 + \int_{\mathcal D_{uv}}F_{12}(U,V)\left(\nabla_x U\cdot\nabla_x V-\frac{V_{x_n}}{U_{x_n}}|\nabla_x U|^2\right)\phi^2,
\end{eqnarray*}
where~$\mathcal R^{n+1}_U$ is as in~\eqref{rU}.
By using a standard differential geometry formula 
(see, for instance formula~$(2.10)$ in~\cite{FSV}), we have
\begin{equation}\begin{split}\label{23}
&\int_{\mathcal R^{n+1}_U}y^{\alpha_1}\left(\sum_{j=1}^n\left(\partial_y U_{x_j}\right)^2 - \left(\partial_y|\nabla_x U|\right)^2\right)\phi^2 \\ 
&\qquad + \int_{\mathcal R^{n+1}_U}y^{\alpha_1}\left(\mathcal K^2|\nabla_x U|^2+\Big|\nabla_{L_U}|\nabla_x U|\Big|^2\right)\phi^2 \\ 
\leq& \int_{\R^{n+1}_+}y^{\alpha_1}|\nabla_x U|^2|\nabla\phi|^2 + \int_{\mathcal D_{uv}}F_{12}(U,V)\left(\nabla_x U\cdot\nabla_x V-\frac{V_{x_n}}{U_{x_n}}|\nabla_x U|^2\right)\phi^2. 
\end{split}\end{equation}
Now, we observe that, on~$\mathcal R^{n+1}_U$, 
\begin{equation}\begin{split}\label{999}
\sum_{j=1}^n\left(\partial_y U_{x_j}\right)^2-\left(\partial_y|\nabla_x U|\right)^2 &= 
|\nabla_x\left(\partial_y U\right)|^2-\left(\partial_y|\nabla_x U|\right)^2 \\
&= |\nabla_x\left(\partial_y U\right)|^2-\Big|\frac{\nabla_x U\cdot\nabla\left(\partial_y U\right)}{\nabla_x U}\Big|^2 \\
&\geq 0,
\end{split}\end{equation}
which implies, together with~\eqref{23}, that 
\begin{eqnarray*}
&& \int_{\mathcal R^{n+1}_U}y^{\alpha_1}\left(\mathcal K_U^2|\nabla_x U|^2+\Big|\nabla_{L_U}|\nabla_x U|\Big|^2\right)\phi^2 \nonumber\\ &\leq& 
\int_{\R^{n+1}_+}y^{\alpha_1}|\nabla_x U|^2|\nabla\phi|^2 + \int_{\mathcal D_{uv}}F_{12}(U,V)\left(\nabla_x U\cdot\nabla_x V-\frac{V_{x_n}}{U_{x_n}}|\nabla_x U|^2\right)\phi^2. 
\end{eqnarray*}
Notice also that~\eqref{borel} and Theorem~$6.19$ in~\cite{LL} give that 
$$ \nabla_x U=0=\nabla_x V,  \mbox{\ almost\ everywhere\ on\ } \mathcal N_{uv}, $$
and therefore
\begin{eqnarray*}
&& \int_{\mathcal D_{uv}}F_{12}(U,V)\left(\nabla_x U\cdot\nabla_x V-\frac{V_{x_n}}{U_{x_n}}|\nabla_x U|^2\right)\phi^2 \\&=& \int_{\partial\R^{n+1}_+}F_{12}(U,V)\left(\nabla_x U\cdot\nabla_x V-\frac{V_{x_n}}{U_{x_n}}|\nabla_x U|^2\right)\phi^2.  
\end{eqnarray*}
This complete the proof of~\eqref{ineq2}. 
In the same way one can prove~\eqref{ineq3}. 
\end{proof}

In order to prove Theorem~\ref{TMon}, 
we take~$\phi=\varphi=\psi$ in~\eqref{ineq2} and~\eqref{ineq3} respectively 
(where~$\varphi$ is as in Theorem~\ref{TMon}) and we sum up the two inequalities to get  
\begin{eqnarray*}
&&\int_{\mathcal R^{n+1}_U}y^{\alpha_1}\left(\mathcal K_U^2|\nabla_x U|^2  +\Big|\nabla_{L_U}|\nabla_x U|\big|^2\right)\varphi^2 \\ 
&&\qquad + \int_{\mathcal R^{n+1}_V}y^{\alpha_2}\left(\mathcal K_V^2|\nabla_x V|^2 +\Big|\nabla_{L_V}|\nabla_x V|\big|^2\right)\varphi^2 \\ 
&\leq& \int_{\R^{n+1}_+}\left(y^{\alpha_1}|\nabla_x U|^2 +y^{\alpha_2}|\nabla_x V|^2\right)|\nabla\varphi|^2 \\ 
&&\qquad + \int_{\R^n}F_{12}(U,V)\left(2\nabla_x U\cdot\nabla_x V -\frac{V_{x_n}}{U_{x_n}}|\nabla_x U|^2-\frac{U_{x_n}}{V_{x_n}}|\nabla_x V|^2\right)\varphi^2 \\ &=& \int_{\R^{n+1}_+}\left(y^{\alpha_1}|\nabla_x U|^2 +y^{\alpha_2}|\nabla_x V|^2\right)|\nabla\varphi|^2 \\ 
&&\qquad + \int_{\R^n}F_{12}(U,V)\left|\sqrt{\frac{-V_{x_n}}{U_{x_n}}}\nabla_x U + \sqrt{\frac{U_{x_n}}{-V_{x_n}}}\nabla_x V\right|^2 \varphi^2, 
\end{eqnarray*}
which is the desired result.

\section{Stable solutions and proof of Theorem \ref{T1}}
In this section we prove the Poincar\'e-type inequality in Theorem~\ref{T1}.

We show first that, under suitable assumptions, a monotone solution of~\eqref{systext} 
is also stable. 
\begin{prop}\label{propMS}
Let~$(U,V)$ be a monotone solution of~\eqref{systext}. 
Suppose 
%that~$U,V$ are~$C^2_{loc}$ in the interior of~$\R^{n+1}_+$, and 
that $F_{12}(U,V)\leq0$. Then~$(U,V)$ is a stable solution. 
\end{prop}

\begin{proof} 
By summing up the inequalities in~\eqref{ineq1}, we have
\begin{eqnarray*} 
0&\leq& \int_{\R^{n+1}_+}\left(y^{\alpha_1}|\nabla\xi_1|^2+y^{\alpha_2}|\nabla\xi_2|^2\right) \\
&&\quad -\int_{\mathcal D_{uv}}\left(F_{11}(U,V)\, \xi_1^2+F_{22}(U,V)\, \xi_2^2+F_{12}(U,V)\left(\frac{V_{x_n}}{U_{x_n}}\, \xi_1^2 +\frac{U_{x_n}}{V_{x_n}}\, \xi_2^2\right)\right) \\
&=& \int_{\R^{n+1}_+}\left(y^{\alpha_1}|\nabla\xi_1|^2+y^{\alpha_2}|\nabla\xi_2|^2\right) \\ && -\int_{\mathcal D_{uv}}\left(F_{11}(U,V)\, \xi_1^2+F_{22}(U,V)\, \xi_2^2 - F_{12}(U,V) \left(\frac{-V_{x_n}}{U_{x_n}}\, \xi_1^2 +\frac{U_{x_n}}{-V_{x_n}}\, \xi_2^2\right)\right) \\
&\leq& \int_{\R^{n+1}_+}\left(y^{\alpha_1}|\nabla\xi_1|^2+y^{\alpha_2}|\nabla\xi_2|^2\right) \\
\\ &&\quad -\int_{\mathcal D_{uv}}\left(F_{11}(U,V)\, \xi_1^2 + F_{22}(U,V)\, \xi_2^2 + 2F_{12}(U,V)\, \xi_1\, \xi_2\right), 
\end{eqnarray*}
where we have used the monotonicity condition, the fact that~$F_{12}(U,V)\leq0$ 
together with 
$$ 0\leq \left(\sqrt{\frac{-V_{x_n}}{U_{x_n}}}\xi_1+\sqrt{\frac{U_{x_n}}{-V_{x_n}}}\xi_2\right)^2 
= \frac{-V_{x_n}}{U_{x_n}}\, \xi_1^2 + \frac{U_{x_n}}{-V_{x_n}}\, \xi_2^2+ 2\, \xi_1\, \xi_2. $$
This concludes the proof. 
\end{proof}

Now, thanks to Lemma~\ref{lemdensity}, we have that~\eqref{estU} and~\eqref{estV} 
hold for~$\phi=U_{x_j}\varphi^2$ and~$\psi=V_{x_j}\varphi^2$ respectively, 
where~$\varphi$ is as in the statement of Theorem~\ref{T1}. 
Therefore, we have
\begin{equation}\begin{split}\label{estU1}
& \int_{\partial B_R^+\cap\mathcal D_{uv}}\left(F_{11}(U,V)|\nabla_x U|^2 + F_{12}(U,V)\nabla_x U\cdot\nabla_x V\right) \varphi^2 \\
=& \sum_{j=1}^n\int_{B_R^+}y^{\alpha_1}\nabla U_{x_j}\cdot\nabla(U_{x_j}\varphi^2) \\ 
=& \sum_{j=1}^n\int_{B_R^+}y^{\alpha_1}\left(|\nabla U_{x_j}|^2\varphi^2 + U_{x_j}\nabla U_{x_j}\cdot \nabla(\varphi^2)\right) \\
=& \int_{B_R^+}y^{\alpha_1}\left(\sum_{j=1}^n |\nabla U_{x_j}|^2\varphi^2 + \varphi\nabla\varphi\cdot\nabla(|\nabla_x U|^2)\right). 
\end{split}\end{equation}
In the same way, we have
\begin{equation}\begin{split}\label{estV1}
& \int_{\partial B_R^+\cap\mathcal D_{uv}}\left(F_{12}(U,V)\nabla_x U\cdot\nabla_x V + F_{22}(U,V)|\nabla_x V|^2\right)\varphi^2 \\
=& \int_{B_R^+}y^{\alpha_2}\left(\sum_{j=1}^n|\nabla V_{x_j}|^2\varphi^2 + \varphi\nabla\varphi\cdot\nabla(|\nabla_x V|^2)\right). 
\end{split}\end{equation} 
By summing up~\eqref{estU1} and~\eqref{estV1}, we obtain
\begin{equation}\begin{split}\label{estUV}
& \int_{B_R^+}y^{\alpha_1}\left(\sum_{j=1}^n|\nabla U_{x_j}|^2\varphi^2 + \varphi\nabla\varphi\cdot\nabla(|\nabla_x U|^2)\right) \\
&\qquad + \int_{B_R^+}y^{\alpha_2}\left(\sum_{j=1}^n|\nabla V_{x_j}|^2\varphi^2 + \varphi\nabla\varphi\cdot\nabla(|\nabla_x V|^2)\right) \\
= & \int_{\partial B_R^+\cap\mathcal D_{uv}}\left[F_{11}(U,V)|\nabla_x U|^2 + F_{22}(U,V)|\nabla_x V|^2 + 2F_{12}(U,V)\nabla_x U\cdot\nabla_x V\right]\varphi^2.
\end{split}\end{equation}
 
Now, we take~$\xi_1:=|\nabla_x U|\varphi$ 
and~$\xi_2:=|\nabla_x V|\varphi$ in~\eqref{stable1}. 
We observe that~\eqref{xi} is satisfied, thanks to~\eqref{estnuova} and~\eqref{L1bis},  
and therefore we can use here such test functions. 
Then, we obtain
\begin{eqnarray*}
0 &\leq& \int_{B_R^+}y^{\alpha_1}|\nabla(|\nabla_x U|\varphi)|^2 + \int_{B_R^+}y^{\alpha_2}|\nabla(|\nabla_x V|\varphi)|^2 \nonumber\\ 
&&\quad - \int_{\partial B_R^+\cap\mathcal D_{uv}}\Big[F_{11}(U,V)|\nabla_x U|^2 + F_{22}(U,V)|\nabla_x V|^2 \nonumber\\ 
&&\qquad\qquad +2F_{12}(U,V) |\nabla_x U|\cdot|\nabla_x V|\Big]\varphi^2 \nonumber\\ 
&=& \int_{B_R^+}y^{\alpha_1}\left[ |\nabla(|\nabla_x U|)|^2\varphi^2 + |\nabla_x U|^2 |\nabla\varphi|^2 + \varphi\nabla\varphi\cdot\nabla(|\nabla_x U|^2)\right] \nonumber\\ 
&& \qquad + \int_{B_R^+}y^{\alpha_2}\left[ |\nabla(|\nabla_x V|)|^2\varphi^2 + |\nabla_x V|^2 |\nabla\varphi|^2 + \varphi\nabla\varphi\cdot\nabla(|\nabla_x V|^2)\right] \nonumber\\ &&\qquad - \int_{\partial B_R^+\cap\mathcal D_{uv}}\Big[F_{11}(U,V)|\nabla_x U|^2 + F_{22}(U,V)|\nabla_x V|^2 \nonumber\\ 
&&\qquad\qquad +2F_{12}(U,V)|\nabla_x U|\cdot|\nabla_x V|\Big]\varphi^2.
\end{eqnarray*}
The last inequality and~\eqref{estUV} imply 
\begin{eqnarray}\label{estUV2}
&& \int_{\R^{n+1}_+}y^{\alpha_1}\left(\sum_{j=1}^n(\partial_y U_{x_j})^2 - (\partial_y|\nabla_x U|)^2 + \sum_{j=1}^n |\nabla_x U_{x_j}|^2 - \Big|\nabla_x|\nabla_x U|\Big|^2\right)\varphi^2 \nonumber\\
&& + \int_{\R^{n+1}_+}y^{\alpha_2}\left(\sum_{j=1}^n(\partial_y V_{x_j})^2 - (\partial_y|\nabla_x V|)^2 + \sum_{j=1}^n |\nabla_x V_{x_j}|^2 - \Big|\nabla_x|\nabla_x V|\Big|^2\right)\varphi^2 \nonumber\\ 
&=& \int_{\R^{n+1}_+}y^{\alpha_1}\left(\sum_{j=1}^n|\nabla U_{x_j}|^2-\Big|\nabla|\nabla_x U|\Big|^2\right)\varphi^2 \nonumber\\ && \qquad +\int_{\R^{n+1}_+}\left(\sum_{j=1}^n|\nabla V_{x_j}|^2-\Big|\nabla|\nabla_x V|\Big|^2\right)\varphi^2 \nonumber\\ &\leq& \int_{\R^{n+1}_+}\left(y^{\alpha_1}|\nabla_x U|^2 +y^{\alpha_2}|\nabla_x V|^2\right)|\nabla\varphi|^2 \nonumber\\ &&\qquad - \int_{\partial\R^{n+1}_+\cap\mathcal D_{uv}} F_{12}(U,V)\left(|\nabla_x U|\cdot|\nabla_x V|-\nabla_x U\cdot\nabla_x V\right)\varphi^2. 
\end{eqnarray}
Arguing exactly as in Theorem~\ref{Tineq} (see the comments after formula~\eqref{22}) 
we have that~$\nabla|\nabla_x U|=0=\nabla U_{x_j}$ for almost any~$(x,y)$ such 
that~$\nabla_x U(x,y)=0$ and that~$\nabla|\nabla_x V|=0=\nabla V_{x_j}$ for almost 
any~$(x,y)$ such that~$\nabla_x V(x,y)=0$.

Hence, we can write~\eqref{estUV2} as
\begin{eqnarray*}
&& \int_{\mathcal R^{n+1}_U}y^{\alpha_1}\left(\sum_{j=1}^n(\partial_y U_{x_j})^2- (\partial_y |\nabla_x U|)^2\right)\varphi^2 \\ &&\qquad + \int_{\mathcal R^{n+1}_U}y^{\alpha_1}\left(\sum_{j=1}^n|\nabla_x U_{x_j}|^2-\Big|\nabla_x|\nabla_x U|\Big|^2\right)\varphi^2 \\ && \qquad +
\int_{\mathcal R^{n+1}_V}y^{\alpha_2}\left(\sum_{j=1}^n(\partial_y V_{x_j})^2- (\partial_y |\nabla_x V|)^2\right)\varphi^2 \\&&\qquad + \int_{\mathcal R^{n+1}_V}y^{\alpha_2}\left(\sum_{j=1}^n|\nabla_x V_{x_j}|^2-\Big|\nabla_x|\nabla_x V|\Big|^2\right)\varphi^2 \\ &\leq& \int_{\R^{n+1}_+}\left(y^{\alpha_1}|\nabla_x U|^2 +y^{\alpha_2}|\nabla_x V|^2\right)|\nabla\varphi|^2 \nonumber\\ &&\qquad - \int_{\mathcal D_{uv}} F_{12}(U,V)\left(|\nabla_x U|\cdot|\nabla_x V|-\nabla_x U\cdot\nabla_x V\right)\varphi^2.
\end{eqnarray*}

By using again a standard differential geometry formula 
(see, for example, formula~$(2.10)$ in~\cite{FSV}), we have
\begin{equation}\begin{split}\label{estUV3}
& \int_{\mathcal R^{n+1}_U}y^{\alpha_1}\left(\sum_{j=1}^n(\partial_y U_{x_j})^2- (\partial_y |\nabla_x U|)^2\right)\varphi^2 \\ &\qquad + \int_{\mathcal R^{n+1}_U}y^{\alpha_1}\left(\mathcal K_U^2|\nabla_x U|^2 +\Big|\nabla_{L_U}|\nabla_x U|\Big|^2 \right)\varphi^2 \\ & \qquad +
\int_{\mathcal R^{n+1}_V}y^{\alpha_2}\left(\sum_{j=1}^n(\partial_y V_{x_j})^2- (\partial_y |\nabla_x V|)^2\right)\varphi^2 \\&\qquad + \int_{\mathcal R^{n+1}_V}y^{\alpha_2}\left(\mathcal K_V^2|\nabla_x V|^2 + \Big|\nabla_{L_V}|\nabla_x V|\Big|^2  \right)\varphi^2 \\ \leq& \int_{\R^{n+1}_+}\left(y^{\alpha_1}|\nabla_x U|^2 +y^{\alpha_2}|\nabla_x V|^2\right)|\nabla\varphi|^2 \\ &\qquad - \int_{\mathcal D_{uv}} F_{12}(U,V)\left(|\nabla_x U|\cdot|\nabla_x V|-\nabla_x U\cdot\nabla_x V\right)\varphi^2.
\end{split}\end{equation}

Recalling~\eqref{999}, we have that, on~$\mathcal R^{n+1}_U$, 
$$
\sum_{j=1}^n (\partial_y U_{x_j})^2- (\partial_y |\nabla_x U|)^2 \geq 0.
$$
In the same way, one can see that, on~$\mathcal R^{n+1}_V$, 
$$
\sum_{j=1}^n (\partial_y V_{x_j})^2- (\partial_y |\nabla_x V|)^2 \geq 0.
$$
The last two inequalities and~\eqref{estUV3} imply that 
\begin{equation}\begin{split}\label{estUV4}
& \int_{\mathcal R^{n+1}_U}y^{\alpha_1}\left(\mathcal K_U^2|\nabla_x U|^2 +\Big|\nabla_{L_U}|\nabla_x U|\Big|^2 \right)\varphi^2 \\&\qquad + \int_{\mathcal R^{n+1}_V}y^{\alpha_2}\left(\mathcal K_V^2|\nabla_x V|^2 + \Big|\nabla_{L_V}|\nabla_x V|\Big|^2  \right)\varphi^2 \\ \leq& \int_{\R^{n+1}_+}\left(y^{\alpha_1}|\nabla_x U|^2 +y^{\alpha_2}|\nabla_x V|^2\right)|\nabla\varphi|^2 
\\ &\qquad - \int_{\mathcal D_{uv}} F_{12}(U,V)\left(|\nabla_x U|\cdot|\nabla_x V|-\nabla_x U\cdot\nabla_x V\right)\varphi^2.
\end{split}\end{equation}
Notice also that~\eqref{borel} and Theorem~$6.19$ in~\cite{LL} give that 
$$ \nabla_x U =0=\nabla_x V, \mathrm{\ almost\ everywhere\ on\ }\mathcal N_{uv}, $$
and therefore
\begin{eqnarray*}
&& \int_{\mathcal D_{uv}} F_{12}(U,V)\left(|\nabla_x U|\cdot|\nabla_x V|-\nabla_x U\cdot\nabla_x V\right)\varphi^2 \\ &=& 
\int_{\partial\R^{n+1}_+} F_{12}(U,V)\left(|\nabla_x U|\cdot|\nabla_x V|-\nabla_x U\cdot\nabla_x V\right)\varphi^2.
\end{eqnarray*}
This and~\eqref{estUV4} complete the proof of Theorem~\ref{T1}.

\section{Proof of Theorem \ref{T2}}

In order to prove Theorem~\ref{T2} we will test the geometric formula in~\eqref{geom} 
against a suitable test function in such a way that the left-hand side vanishes. 
Hence, this will imply that the tangential gradient and 
the curvature of the level sets of the functions~$U$ and~$V$, 
for fixed~$y>0$, vanish. The conclusion will be that these level sets are flat, 
as we desire. 

In the sequel we will denote by~$X:=(x,y)$ the points in~$\R^{n+1}_+$. 

For any~$\rho_1\leq\rho_2$, we define
$$ \mathcal B_{\rho_1,\rho_2}:=\left\lbrace X\in\R^{n+1}_+ : |X|\in[\rho_1,\rho_2]\right\rbrace . $$

We have the following lemma (see Lemma~$10$ in~\cite{SV} for a simple proof):
\begin{lemma}\label{lemh} 
Let~$R>0$ and~$h:B_R^+\rightarrow\R$ be a nonnegative measurable function. 

For any~$\rho\in(0,R)$, let 
$$ \eta(\rho):=\int_{B_{\rho}^+} h. $$

Then, 
$$ \int_{\mathcal B_{\sqrt R,R}}\frac{h(X)}{|X|^2}\, dX \leq 2\int_{\sqrt R}^R \frac{\eta(t)}{t^3}\, dt + \frac{\eta(R)}{R^2}. $$
\end{lemma}

We will deduce Theorem~\ref{T2} from the following symmetry result, 
which holds in any dimension~$n$. 
\begin{theorem}\label{Tn}
Let the assumptions of Corollary~\ref{cor1} hold. Suppose that~$U,V$ are bounded. 
Moreover, assume that there exists~$C_0\geq1$ such that 
\begin{equation}\label{R2}
\int_{B_R^+} y^{\alpha_1}|\nabla U|^2\leq C_0 R^2, \qquad \int_{B_R^+} y^{\alpha_2}|\nabla V|^2\leq C_0 R^2
\end{equation}
for any~$R\geq C_0$. 

Then, there exist~$\omega_U,\omega_V\in S^{n-1}$, 
and~$U_0,V_0:\R\times(0,+\infty)\rightarrow\R$ such that 
\begin{equation}\label{omegaUV}
 U(x,y)=U_0(\omega_U\cdot x,y), \qquad V(x,y)=V_0(\omega_V\cdot x,y) 
\end{equation}
for any~$(x,y)\in\R^{n+1}_+$. 
\end{theorem}

\begin{proof} 
We apply Lemma~\ref{lemh} 
with~$h(X)=y^{\alpha_1}|\nabla U(X)|^2+y^{\alpha_2}|\nabla V(X)|^2$ and we use~\eqref{R2} 
to obtain that 
\begin{equation}\label{logR}
\int_{\mathcal B_{\sqrt R,R}}\frac{y^{\alpha_1}|\nabla U(X)|^2+y^{\alpha_2}|\nabla V(X)|^2}{|X|^2} \leq C_1 \log R
\end{equation}
for a suitable~$C_1$, and for~$R$ large enough. 

Now, we chose conveniently~$\varphi$ in~\eqref{geom}. 
For any~$R>1$, we define the function~$\varphi_R$ as  
\begin{equation}\label{fiR}
\varphi_R(X):=\left\{
\begin{matrix}
1 & {\mbox{ if $|X|\leq\sqrt R$,}} \\
2\, \frac{\log R-\log|X|}{\log R}
& {\mbox{ if $\sqrt R\leq |X|\leq R$,}}\\
0 & {\mbox{ if $|X|\geq R$.}}
\end{matrix}
\right.
\end{equation}
Notice that
$$  |\nabla\varphi_R(X)|\leq C_2\, \frac{\chi_{\mathcal B_{\sqrt R,R}}}{|X|\log R}. $$

Now, if we plug~$\varphi_R$ in the geometric inequality~\eqref{geom} 
and we use~\eqref{logR}, we have that, for large~$R$,
\begin{equation}\begin{split}\label{logR2}
& \int_{\mathcal R^{n+1}_U\cap B_{\sqrt R}^+}y^{\alpha_1}\left(\mathcal K_U^2 |\nabla_x U|^2 + \Big|\nabla_{L_U}|\nabla_x U|\Big|^2\right) \\ &\qquad + 
\int_{\mathcal R^{n+1}_V\cap B_{\sqrt R}^+}y^{\alpha_2}\left(\mathcal K_V^2 |\nabla_x V|^2 + \Big|\nabla_{L_V}|\nabla_x V|\Big|^2\right) \\
\leq& \frac{C_3}{(\log R)^2}\int_{\mathcal B_{\sqrt R,R}}\frac{y^{\alpha_1}|\nabla_x U|^2 + y^{\alpha_2}|\nabla_x V|^2}{|X|^2} \\ \leq& 
\frac{C_4}{\log R}.   
\end{split}\end{equation}
Letting~$R\rightarrow +\infty$ in~\eqref{logR2}, 
we obtain that 
\begin{eqnarray*}
&&\mathcal K_U^2 |\nabla_x U|^2 + \Big|\nabla_{L_U}|\nabla_x U|\Big|^2 =0 \qquad \mathrm{on\ } \mathcal R^{n+1}_U, \\
&&\mathcal K_V^2 |\nabla_x V|^2 + \Big|\nabla_{L_V}|\nabla_x V|\Big|^2 =0 \qquad \mathrm{on\ } \mathcal R^{n+1}_V, 
\end{eqnarray*}
which implies that 
\begin{eqnarray*}
&& \mathcal K_U =0= \Big|\nabla_{L_U}|\nabla_x U|\Big| \qquad \mathrm{on\ } \mathcal R^{n+1}_U, \\
&& \mathcal K_U =0=\Big|\nabla_{L_V}|\nabla_x V|\Big| \qquad \mathrm{on\ }\mathcal R^{n+1}_V. 
\end{eqnarray*}

Then, from Lemma~$2.11$ in~\cite{FSV} or Lemma~$5$ in~\cite{Di}, it follows that 
there exist~$\omega_U,\omega_V:(0,+\infty)\rightarrow S^{n-1}$ 
and~$U_0,V_0:\R\times(0,+\infty)\rightarrow\R$ such that 
\begin{equation}\begin{split}\label{1D}
U(x,y) &= U_0(\omega_U(y)\cdot x,y), \\
V(x,y) &= V_0(\omega_V(y)\cdot x, y),
\end{split}\end{equation}
for any~$(x,y)\in\R^{n+1}_+$. 

Now, we use the fact that~$S^{n-1}$ is compact, to say that we can take a sequence~$y_j\rightarrow0^+$ 
and~$\omega_U,\omega_V\in S^{n-1}$ in such a way 
that~$\omega^j_U:=\omega_U(y_j)\rightarrow\omega_U$ 
and~$\omega^j_V:=\omega_V(y_j)\rightarrow\omega_V$. 
Hence, by~\eqref{1D} and Lemma~\ref{lem:regUV}, we have
\begin{equation*}\begin{split}
u(x)=U(x,0)=\lim_{j\rightarrow +\infty}U(x,y_j) = \lim_{j\rightarrow +\infty}U_0(\omega_U^j\cdot x,y_j) = u_0(\omega_U\cdot x), \\
v(x)=V(x,0)=\lim_{j\rightarrow +\infty}V(x,y_j) = \lim_{j\rightarrow +\infty}V_0(\omega_V^j\cdot x,y_j) = v_0(\omega_V\cdot x). 
\end{split}\end{equation*}

As in~\cite{CS} and~\cite{SV}, we consider the Poisson kernel defined in~\eqref{poisson}. Then, we also define 
\begin{eqnarray*}
U^*(x,y):=\int_{\R^n}P_{\alpha_1}(\zeta,y)u(x-\zeta) \, d\zeta = \int_{\R^n}P_{\alpha_1}(\zeta,y) u_0(\omega_U\cdot x-\omega_U\cdot\zeta)\, d\zeta, \\
V^*(x,y):=\int_{\R^n}P_{\alpha_2}(\zeta,y)v(x-\zeta) \, d\zeta = \int_{\R^n}P_{\alpha_2}(\zeta,y) v_0(\omega_V\cdot x-\omega_V\cdot\zeta)\, d\zeta.
\end{eqnarray*}
Now, from the above definitions of~$U^*$ and~$V^*$, we have that there exist 
functions $U^*_0,V^*_0$ such that~$U^*(x,y)=U^*_0(\omega_U\cdot x,y)$ 
and~$V^*(x,y)=V^*_0(\omega_V\cdot x,y)$. 

Now, we define the functions~$\overline U:=U-U^*$ and~$\overline V:=V-V^*$. 
We observe that
$$ div(y^{\alpha_1}\nabla\overline U)=0=div(y^{\alpha_2}\nabla\overline V) $$ 
in $\R^{n+1}_+$, 
thanks to~\cite{CS} (recall Lemma~\ref{lemP}). 
Moreover, since~$U$ and~$V$ are bounded, we have that~$\overline U$ and~$\overline V$ are also bounded. Finally, we have also that~$\overline U(x,0)=0=\overline V(x,0)$. 

Therefore, by a Liouville-type result (see, for instance, the footnote~$3$ in~\cite{SV} 
or p.~$431$ in~\cite{CSS}), we obtain that~$\overline U$ and~$\overline V$ vanish 
identically. 

Hence, 
$$ U(x,y)=U^*_0(\omega_U\cdot x,y), \qquad V(x,y)=V^*_0(\omega_V\cdot x,y), $$
which gives~\eqref{omegaUV}.
\end{proof}

In order to complete the proof of Theorem~\ref{T2}, 
we notice that, since the condition~\eqref{condy} is satisfied under the assumptions of Theorem~\ref{T2}, 
the estimates in~\eqref{R2} hold true. 
Therefore, the hypotheses of Theorem~\ref{T2} imply the ones of Theorem~\ref{Tn}, 
and then we get the desired conclusion.

\section{Proof of Theorem \ref{T3}}
We will deduce Theorem~\ref{T3} from Theorem~\ref{T2}. 

First, we would like to notice that, given a function~$u$, 
the extension is not, in general, unique. 
In fact, for example, one can consider the functions~$u:=0$ and~$u:=y^{1-\alpha}$; 
then, they both satisfy~$div(y^{\alpha}\nabla u)=0$ in~$\R^{n+1}_+$ with~$u=0$ on~$\partial\R^{n+1}_+$.

To prove Theorem~\ref{T3}, given functions~$u,v$ satisfying~\eqref{systfrac}, 
we choose extensions~$U,V$ satisfying~\eqref{systext} 
by the Poisson kernel in~\eqref{poisson}. 

Hence, if~$u,v$ are bounded solutions to~\eqref{systfrac}, 
we consider the functions defined in~\eqref{extU} and~\eqref{extV}. 
We recall that~$U,V$ are bounded in~$\R^{n+1}_+$ if~$u,v$ are bounded in~$\R^n$ 
(see Section~\ref{sec:reg}, the comments after Lemma~\ref{lemP}). 

Next we prove a regularity result. 
\begin{lemma}\label{lemreg} 
Let~$u,v$ be bounded and~$C^2_{loc}(\R^n)$, 
and let~$U,V$ be given by~\eqref{extU} and~\eqref{extV}. 

Then, for any~$R>0$ there exists~$C_R>0$ such that 
$$ \left\|y^{\alpha_1}\partial_y U\right\|_{L^{\infty}(\overline{B_R^+})}\leq C_R, 
\qquad  \left\|y^{\alpha_2}\partial_y V\right\|_{L^{\infty}(\overline{B_R^+})}\leq C_R. $$ 
\end{lemma}

\begin{proof} 
We prove the estimate for~$U$, in the same way one obtains also the estimate for~$V$. 

Since
$$ \int_{\R^n}P_{\alpha_1}(x,y) dy =1, $$
(see Lemma~\ref{lemP}), we can write
\begin{eqnarray*}
U(x,y)-u(x) &=& \int_{\R^n}P_{\alpha_1}(x-\zeta,y)\left[ u(x-\zeta)-u(x)\right] d\zeta \\
&=& C_{n,\alpha_1}\int_{\R^n} \frac{y^{1-\alpha_1}\left[u(x-\zeta)-u(x)\right]}{\left(|\zeta|^2+y^2\right)^{\frac{n+1-\alpha_1}{2}}} d\zeta . 
\end{eqnarray*}
Therefore, 
\begin{equation}\begin{split}\label{60}
y^{\alpha_1}\partial_y U(x,y) &= y^{\alpha_1}\partial_y \left(U(x,y)-u(x)\right) \\
&= C_{n,\alpha_1}\int_{\R^n} \frac{\left[(1-\alpha_1)|\zeta|^2-ny^2\right]\left[u(x-\zeta)-u(x)\right]}{\left(|\zeta|^2+y^2\right)^{\frac{n+3-\alpha_1}{2}}} \, d\zeta \\
&= C_{n,\alpha_1} \int_{|\zeta|\leq1}\frac{\left[(1-\alpha_1)|\zeta|^2-ny^2\right]\left[u(x-\zeta)-u(x)\right]}{\left(|\zeta|^2+y^2\right)^{\frac{n+3-\alpha_1}{2}}}\, d\zeta \\ 
&\quad + C_{n,\alpha_1} \int_{|\zeta|\geq1}\frac{\left[(1-\alpha_1)|\zeta|^2-ny^2\right]\left[u(x-\zeta)-u(x)\right]}{\left(|\zeta|^2+y^2\right)^{\frac{n+3-\alpha_1}{2}}}\, d\zeta.
\end{split}\end{equation}

We estimate the first integral in the right-hand side of~\eqref{60}. 
Since the functions 
$$\frac{(1-\alpha_1)|\zeta|^2\zeta}{\left(|\zeta|^2+y^2\right)^{\frac{n+3-\alpha_1}{2}}}, \qquad \frac{ny^2\zeta}{\left(|\zeta|^2+y^2\right)^{\frac{n+3-\alpha_1}{2}}} $$
are odd with respect to~$\zeta$, we have that
$$\int_{|\zeta|\leq1}\frac{(1-\alpha_1)|\zeta|^2\nabla u(x)\cdot\zeta}{\left(|\zeta|^2+y^2\right)^{\frac{n+3-\alpha_1}{2}}}\, d\zeta=0 = \int_{|\zeta|\leq1}\frac{ny^2\nabla u(x)\cdot\zeta}{\left(|\zeta|^2+y^2\right)^{\frac{n+3-\alpha_1}{2}}}\, d\zeta . $$
Therefore, we can write 
\begin{eqnarray*}
&&\int_{|\zeta|\leq1}\frac{\left[(1-\alpha_1)|\zeta|^2-ny^2\right]\left[u(x-\zeta)-u(x)\right]}{\left(|\zeta|^2+y^2\right)^{\frac{n+3-\alpha_1}{2}}}\, d\zeta \nonumber\\ &=& 
\int_{|\zeta|\leq1}\frac{\left[(1-\alpha_1)|\zeta|^2-ny^2\right]\left[u(x-\zeta)-u(x)\right]}{\left(|\zeta|^2+y^2\right)^{\frac{n+3-\alpha_1}{2}}}\, d\zeta \nonumber\\&&\qquad -\int_{|\zeta|\leq1}\frac{\left[(1-\alpha_1)|\zeta|^2-ny^2\right]\nabla u(x)\cdot\zeta}{\left(|\zeta|^2+y^2\right)^{\frac{n+3-\alpha_1}{2}}}\, d\zeta. 
\end{eqnarray*}
Hence, we get 
\begin{equation}\begin{split}\label{62}
&\int_{|\zeta|\leq1}\frac{\left[(1-\alpha_1)|\zeta|^2-ny^2\right] \left[u(x-\zeta)-u(x)\right]}{\left(|\zeta|^2+y^2\right)^{\frac{n+3-\alpha_1}{2}}}\, d\zeta \\  \leq & 
\int_{|\zeta|\leq1}\frac{\left[(1-\alpha_1)|\zeta|^2-ny^2\right]\, |u(x-\zeta)-u(x)-\nabla u(x)\cdot\zeta|}{\left(|\zeta|^2+y^2\right)^{\frac{n+3-\alpha_1}{2}}}\, d\zeta \\ \leq & 
\int_{|\zeta|\leq1}\frac{\left[(1-\alpha_1)|\zeta|^2-ny^2\right]|D^2 u(x)|\,|\zeta|^2}{\left(|\zeta|^2+y^2\right)^{\frac{n+3-\alpha_1}{2}}}\, d\zeta \\ \leq & C
\int_{|\zeta|\leq1}\frac{|D^2 u(x)|\,|\zeta|^2}{\left(|\zeta|^2+y^2\right)^{\frac{n+1-\alpha_1}{2}}}\, d\zeta \\ \leq&  C
\int_{|\zeta|\leq1}\frac{|D^2 u(x)|}{|\zeta|^{n-1-\alpha_1}}\, d\zeta,
\end{split}\end{equation}
which is summable. 

Now, we estimate the second integral in the right-hand side of~\eqref{60}:
\begin{equation}\begin{split}\label{63}
&\int_{|\zeta|\geq1}\frac{\left[(1-\alpha_1)|\zeta|^2-ny^2\right]\left[u(x-\zeta)-u(x)\right]}{\left(|\zeta|^2+y^2\right)^{\frac{n+3-\alpha_1}{2}}}\, d\zeta \\ 
\leq & C
\int_{|\zeta|\geq1}\frac{|u(x-\zeta)-u(x)|}{\left(|\zeta|^2+y^2\right)^{\frac{n+1-\alpha_1}{2}}}\, d\zeta \\  \leq & C
\int_{|\zeta|\geq1} \frac{2\|u\|_{L^{\infty}(\R^n)}}{|\zeta|^{n+1-\alpha_1}}\, d\zeta,
\end{split}\end{equation}
which is again summable. 

Putting together~\eqref{60},~\eqref{62} and~\eqref{63} we obtain the bound 
$$ \left\|y^{\alpha_1}\partial_y U\right\|_{L^{\infty}(\overline{B_R^+})}\leq C\left(\|u\|_{L^{\infty}(\R^n)} + 
\|D^2 u\|_{L^{\infty}(B_{R+1})} \right) $$
as desired. 
\end{proof}

Now we give the proof of Theorem~\ref{T3}.  
We take~$U,V$ as defined in~\eqref{extU} and~\eqref{extV} and 
we notice that~\eqref{cond} is satisfied.
Indeed, thanks to 
the local integrability of~$y^{\alpha_1}, y^{\alpha_2},y^{-\alpha_1}, y^{-\alpha_2}$, 
Lemma~\ref{lemreg} and Proposition~\ref{prop:reg}, we have
\begin{eqnarray*}
\int_{B_R^+}y^{\alpha_1}|\nabla U|^2 &=& \int_{B_R^+}y^{\alpha_1}|\partial_y U|^2 + 
\int_{B_R^+}y^{\alpha_1}|\nabla_x U|^2 \\ &\leq& \int_{B_R^+}y^{2\alpha_1}|\partial_y U|^2 y^{-\alpha_1} + C_R \int_{B_R^+}y^{\alpha_1} \\ &\leq& C_R\int_{B_R^+}y^{-\alpha_1} + C_R \int_{B_R^+}y^{\alpha_1} \\ &\leq& C_R,
\end{eqnarray*}
and the same for~$y^{\alpha_2}|\nabla V|^2$. 

Also, we know that either~$(U,V)$ is monotone and~$F_{12}(U,V)\leq0$ 
or~$(U,V)$ is stable, thanks to~\eqref{stable1}, and~$F_{12}(U,V)\geq0$. 

Then, from Theorem~\ref{T2} we have that there exist functions~$U_0$ and~$V_0$, 
and directions~$\omega_U,\omega_V$ such that   
$$ U(x,y)=U_0(\omega_U\cdot x,y), \qquad V(x,y)=V_0(\omega_V\cdot x,y) $$
for any~$x\in\R^2$ and any~$y>0$.

Now, from Lemma~\ref{lem:regUV}, we deduce that~$U,V$ are continuous 
up to~$\left\lbrace y=0\right\rbrace$, 
and then 
$$ U(x,0)=U_0(\omega_U\cdot x,0), \qquad V(x,0)=V_0(\omega_V\cdot x,0). $$

Since, by~\eqref{eqpoisson},~\eqref{extU} and~\eqref{extV}, 
$$ U|_{\partial\R^{n+1}_+}=u, \qquad V|_{\partial\R^{n+1}_+}=v, $$ 
the proof of the first part of Theorem~\ref{T3} is complete. 

\medskip 
We prove next the second part. 
Suppose first that~\eqref{monF12} holds. 
From Theorem~\ref{T2} we know that~$(U,V)$ has one-dimensional symmetry. 
Therefore, by using the inequality in Theorem~\ref{TMon}, we have that 
\begin{eqnarray*}
&& -\int_{\R^n}F_{12}(U,V)\left|\sqrt{\frac{-V_{x_n}}{U_{x_n}}}\nabla_x U + \sqrt{\frac{U_{x_n}}{-V_{x_n}}}\nabla_x V\right|^2\varphi^2  \\ &\leq& 
\int_{\R^{n+1}_+}\left(y^{\alpha_1}|\nabla_x U|^2+y^{\alpha_2}|\nabla_x V|^2\right) |\nabla\varphi|^2. 
\end{eqnarray*}
Choosing~$\varphi_R$ as in~\eqref{fiR} and reasoning as in the proof of Theorem~\ref{Tn}, 
we obtain 
$$ F_{12}(U,V)\left|\sqrt{\frac{-V_{x_n}}{U_{x_n}}}\nabla_x U + \sqrt{\frac{U_{x_n}}{-V_{x_n}}}\nabla_x V\right|^2=0, \qquad \mbox{for\ a.e.\ } x\in\R^n, $$
since $F_{12}(U,V)\leq0$. 

By~\eqref{monF12} we have that there exists~$\bar x\in\R^n$ such 
that~$F_{12}(U(\bar x,0),V(\bar x,0))<0$. Therefore, 
$$ \sqrt{\frac{-V_{x_n}(\bar x,0)}{U_{x_n}(\bar x,0)}}\nabla_x U(\bar x,0) + \sqrt{\frac{U_{x_n}(\bar x,0)}{-V_{x_n}(\bar x,0)}}\nabla_x V(\bar x,0)=0, $$
which gives that 
\begin{equation}\label{-999}
\nabla_x U(\bar x,0)=h(\bar x)\nabla_x V(\bar x,0),
\end{equation} 
for some function~$h$. 
Since we know that~$(U,V)$ has a one-dimensional symmetry, 
we have that~$\nabla_x U(\bar x,0)$ is proportional to~$\omega_U$ 
and~$\nabla_x V(\bar x,0)$ is proportional to~$\omega_V$. 
Therefore,~\eqref{-999} implies that 
\begin{equation}\label{-0}
\omega_U=\pm \omega_V.
\end{equation}
Now, if~$\omega_U=-\omega_V$, 
$u(x)=u_0(\omega_U\cdot x)$ and~$v(x)=v_0(\omega_V\cdot x)$,
then we can define~$\tilde v_0(t):=v_0(-t)$ and obtain~$v(x)=\tilde v_0
(\omega_U\cdot x)$ (i.e., from~\eqref{-0} we obtain that
we can choose~$\omega_U=\omega_V$ up to renaming the one-dimensional
function that describes~$v$). 
Hence, we have that~$\omega_U=\omega_V$. 

Finally, we assume that~\eqref{stabF12} holds and we show 
that~$\omega_U=\omega_V$. For this, we observe that if~$u$ is constant,
we can choose~$\omega_U$ as we like (and in particular we can 
choose~$\omega_U=\omega_V$). Since the same argument holds for~$v$, we
may assume that
\begin{equation}\label{-1}
{\mbox{both $u$ and $v$ are non-constant,}}\end{equation}
otherwise we are done. 
As in the case of monotone solutions, it is sufficient to prove~\eqref{-0} 
(see comments after~\eqref{-0}). 
Hence, we argue by contradiction and we assume that~\eqref{-0}
is not true, i.e.
\begin{equation}\label{-00}
\omega_U\not=\pm \omega_V.
\end{equation}
Then we claim that 
\begin{equation}\label{-2}
\begin{split}
&{\mbox{there exists
$x_\sharp\in \R^2$ such that $u(x_\sharp)\in I_u$,
$v(x_\sharp)\in I_v$,}} \\
&{\mbox{$\nabla u(x_\sharp)\ne0$
and $\nabla v(x_\sharp)\ne0$. 
}}\end{split}\end{equation}
To prove this, we notice that, since~$(I_u\times I_v)\cap 
\Im(u,v)\ne\varnothing$, there exists~$x_1\in\R^2$ such that~$u(x_1)\in 
I_u$ (and also~$v(x_1)\in I_v$). By~\eqref{-1}, there 
exists~$x_2\in\R^2$ such that~$u(x_2)\ne u(x_1)$,
say~$u(x_2)>u(x_1)$.
Hence, by continuity,
there exists~$x_3$ on the open segment that joins~$x_1$ and~$x_2$
such that~$u(x_3)\in I_u$ and~$u(x_3)>u(x_1)$. 
Therefore, by the Fundamental Theorem of Calculus, there exists~$x_4$ on 
the open segment that joins~$x_1$ 
and~$x_3$ such that~$\nabla u(x_4)\ne0$. Now we denote~$\pi_u$ the 
hyperplane normal to~$\omega_U$ passing through~$x_4$. Since~$u$
has one-dimensional symmetry, we know that $u$ is constant
on~$\pi_u$ with value in~$I_u$,
and~$\nabla u$ is a constant non-zero 
vector
on~$\pi_u$ that is parallel to~$\omega_U$.

By performing a similar argument on~$v$, we obtain that there
exists a hyperplane~$\pi_v$ normal to~$\omega_V$
such that~$v$ is constant
on~$\pi_v$ with value in~$I_v$,
and~$\nabla v$ is a constant non-zero 
vector
on~$\pi_v$ that is parallel to~$\omega_V$.

As a consequence of~\eqref{-00}, $\pi_u$ and~$\pi_v$
must intersect. Let~$x_\sharp\in \pi_u\cap \pi_v$.
Then, since~$x_\sharp\in \pi_u$, we have that~$u(x_\sharp)\in I_u$
and~$\nabla u(x_\sharp)\ne0$, while the fact that~$x_\sharp\in \pi_v$
implies that~$v(x_\sharp)\in I_v$
and~$\nabla v(x_\sharp)\ne0$, thus proving~\eqref{-2}.

By continuity, from~\eqref{-2} we deduce that
\begin{equation}\label{-22}\begin{split}
&{\mbox{there exists a non-empty open set~$\Omega\subset\R^2$
such that}}\\ &{\mbox{$u(x)\in I_u$,
$v(x)\in I_v$, $\nabla u(x)\ne0$
and $\nabla v(x)\ne0$ for all $x\in\Omega$.
}}\end{split}\end{equation}
Now, since we know that $(U,V)$ has one-dimensional symmetry, from the 
inequality in Theorem~\ref{T1} we have that 
\begin{eqnarray*}
&& 2\int_{\partial\R^{n+1}_+} F_{12}(U,V)\left(|\nabla_x U|\cdot |\nabla_x V|-\nabla_x U\cdot\nabla_x V\right)\varphi^2 \\ &\leq& \int_{\R^{n+1}_+}\left(y^{\alpha_1}|\nabla_x U|^2 +y^{\alpha_2}|\nabla_x V|^2\right)|\nabla\varphi|^2. 
\end{eqnarray*} 
Choosing~$\varphi_R$ as in~\eqref{fiR} and reasoning as in the proof of Theorem~\ref{Tn}, 
we obtain
$$ F_{12}(U,V)\left(|\nabla_x U|\cdot |\nabla_x V|-\nabla_x U\cdot\nabla_x V\right)=0, \qquad \mbox{\ for\ a.e.\ } x\in \R^n, $$ 
since~$F_{12}(U,V)\geq0$. 

By~\eqref{-22} we have that there exists~$x_\star \in\Omega$ such 
that~$F_{12}(U(x_\star,0),V(x_\star,0))>0$. Therefore, 
\begin{equation}\begin{split}\nonumber 
& |\nabla_x U(x_\star,0)|\cdot |\nabla_x V(x_\star,0)|-\nabla_x 
U(x_\star,0)\cdot\nabla_x V(x_\star,0)
\\ =& |\nabla_x U(x_\star,0)|\cdot |\nabla_x V(x_\star,0)|\\ 
&\qquad-|\nabla_x 
U(x_\star,0)|\cdot |\nabla_x V(x_\star,0)|\frac{\nabla_x 
U(x_\star,0)}{|\nabla_x 
U(x_\star,0)|}\cdot\frac{\nabla_x V(x_\star,0)}{|\nabla_x V(x_\star,0)|} 
=0.
\end{split}\end{equation}
Since~$\nabla_x U(x_\star,0)=\nabla u(x_\star)\ne0$ and~$\nabla_x 
V(x_\star,0)=\nabla v(x_\star)\ne0$
by~\eqref{-22}, we conclude that
$$ \frac{\nabla_x U(x_\star,0)}{|\nabla_x U(x_\star,0)|}\cdot 
\frac{\nabla_x 
V(x_\star,0)}{|\nabla_x V(x_\star,0)|}=1. $$
Since we know that~$(U,V)$ has a one dimensional symmetry,
we have that~$\nabla_x U(x_\star,0)$ is proportional 
to~$\omega_U$
and~$\nabla_x V(x_\star,0)$ is proportional to~$\omega_V$. We 
obtain that
$$ \omega_U\cdot \omega_V=\pm 1.$$
This and the Cauchy Inequality imply that~$\omega_{U}=\pm \omega_{V}$. 
This concludes the proof of Theorem~\ref{T3}. 

\section*{Acknowledgments} The authors want to thank \emph{Enrico Valdinoci} 
for very helpful discussions and comments and the anonymous Referee for 
his or her deep observations.

\vspace{2mm}

\end{document}